\documentclass[11pt,leqno,a4paper]{article}

\makeatletter
  
  \@addtoreset{equation}{section}
\makeatother

\usepackage{upref,amsxtra,amscd,amsopn,amsbsy,amstext,amssymb,amsmath,amsthm}
\usepackage{bm}
\usepackage{xcolor,esint}
\usepackage{hyperref}

\newtheorem{thm}{Theorem}[section]
\newtheorem{lem}{Lemma}[section]

\newtheorem{rem}{Remark}[section]

\newtheorem{dfn}{Definition}[section]

\newcommand{\pd}{\partial}
\newcommand{\gm}{\gamma}
\newcommand{\Gm}{\Gamma}

\newcommand{\lm}{\lambda}
\newcommand{\R}{\mathbb{R}}

\newcommand{\N}{\mathbb{N}}

\newcommand{\av}[1]{\left| #1 \right|}

\newcommand{\vk}{\kappa}
\newcommand{\vp}{\varphi}

\newcommand{\va}{\alpha}

\newcommand{\ve}{\varepsilon}

\newcommand{\mL}{\mathcal{L}}

\newcommand{\argmin}{\operatornamewithlimits{arg\,min}}

\usepackage{color}

\allowdisplaybreaks

\begin{document}

\title{\bf A gradient flow for the $p$-elastic energy \\ defined on closed planar  curves}
\author{{\sc S. Okabe} \thanks{Mathematical Institute, Tohoku University, Japan, \url{okabes@m.tohoku.ac.jp} }, {\sc P. Pozzi}  \thanks{Fakult\"at f\"ur Mathematik, Universit\"at Duisburg-Essen, Germany, \url{paola.pozzi@uni-due.de}} , and {\sc G. Wheeler} \thanks{Institute for Mathematics and Its Applications, University of Wollongong, Australia, \url{glenw@uow.edu.au}}}
\date{\today}

\maketitle

\begin{abstract}
We study the evolution of closed inextensible planar curves under a second order flow that decreases the $p$-elastic energy.
A short time existence result for $p \in (1, \infty) $ is obtained via a minimizing movements method.
For $p=2$, that is in the case of the classic elastic energy, long-time existence is retrieved.
\end{abstract}


\tableofcontents



\section{Introduction}
\label{sec:F4}
Let $\gm: S^{1} \to \R^{2}$ denote a  smooth  regular closed planar curve and let the  $p$-elastic energy be defined by  
\begin{align} \label{p-energy}
E_{p}(\gm) := \dfrac{1}{p} \int_{\gm} \av{\vk(s)}^{p} \, ds, 
\end{align}
where $\vk$ and $s$ denote the scalar curvature and the arc length parameter of $\gm$, and $ 1<p <\infty$.

In the following we study a \emph{second order gradient flow} for the energy \eqref{p-energy}: our approach follows  ideas of  Wen~\cite{Wen} and Lin-Lue-Schwetlick~\cite{Lin15}, which deal with the cases of closed resp. open curves for $p=2$, and generalizes results of Wen \cite{Wen}  to the cases where $p \neq 2$. However, our Ansatz is different in that we employ minimizing movements ideas to prove short time existence of the flow. In particular we prove existence  of the flow in a \emph{weak sense}.

\bigskip

Before we proceed with giving our main statements let us remark that the  above energy \eqref{p-energy} and its associated gradient flow (for mostly smooth initial data) have been intensively studied in the case $p=2$ (see for instance \cite{LangerSinger1}, \cite{LangerSinger2} \cite{DKS},
\cite{Wen95}, \cite{Koiso}, \cite{Polden}, \cite{Lin12}, \cite{DP14}, \cite{DLCana}, \cite{DCP17}, \cite{DPS}, \cite{DLCPS}, \cite{Spener}, \cite{Okabe2007}, \cite{okabe2008}, \cite{NO14}, \cite{NO17}, \cite{Wheeler}, \cite{ManMar}, \cite{Oelz11}, \cite{Oelz14}  and references given therein). The case $p\neq 2$ has received so far  less attention.
Critical points for  \eqref{p-energy}  and $1<p<\infty$ are studied by Watanabe in \cite{Watanabe}. The generalized elastica problem under area constraint is studied by  Ferone-Kawohl-Nitsch in \cite{FKN}. 

The energy $E_p(\gm)$ is relevant for applications. For instance  it has appeared  in a mathematical model for image processing, as we shall now explain.  
In a  mathematical model for image restoration, Rudin-Osher-Fatemi \cite{ROF} proposed the following variational problem: 
\begin{align} \label{rof-model}
\argmin_{u \in BV(\Omega)}\left\{  \int_{\Omega} |\nabla u| \, dx + \lambda \int_{\Omega} |u-g|^{2}\, dx \right\}, 
\end{align}
where $BV(\Omega)$ denotes the space of functions of bounded variations in a bounded domain $\Omega \subset \R^2$, 
$\lm$ is a positive constant, and $g$ is a given function representing a noisy image. 
A reconstructed version of $g$ is obtained as the minimizer $u$ of \eqref{rof-model}. 
Since the model \eqref{rof-model} does not distinguish between jumps and smooth transitions, 
Chan-Marquina-Mulet \cite{CMM} considered an additional penalization of the discontinuity, proposing the following higher order model for image restoration: 
\begin{align} \label{cmm-model}
\argmin\left\{ \int_{\Omega} \psi(|\nabla u|) h(\Delta u) \, dx + \int_{\Omega} |\nabla u| \, dx + \lambda \int_{\Omega} |u-g|^{2}\, dx \right\}, 
\end{align}
where $\psi$ is a suitably chosen map. For the one-dimensional case $\Omega = (a, b)$, the energy functional in \eqref{cmm-model} with $\psi(t) = \dfrac{1}{(1+t^{2})^{(3p-1)/2}}$ and $h(t)=t^{p}$, is then slightly modified into 
\begin{align} \label{1-d-ir}
  \int_{c_{u}} |\vk_{u}|^{p} \, ds + \mL(c_{u})+\lambda \int^b_a | u-g |^2 \, dx, 
\end{align}
where $c_{u}(t)$ is the curve $(t, u(t))$, and $\mL(c_{u})$ and $\vk_{u}$ denotes the length and the scalar curvature of $c_{u}$ (see \cite{AM, DFLM}). 

\bigskip

Our main results read as follows.
Let $I := (0, L)$ and $\phi(s) := 2 \pi \eta s / L$ for some fixed $\eta \in \N$. 
For maps $u \in W^{1,p}_{\rm per}(I)$ we   consider the problem
\begin{align} \label{eq:P} \tag{P}
\begin{cases}
&\partial_t u = \partial_s \left( |\partial_s(u + \phi)|^{p-2} \partial_s(u+\phi) \right) + \tilde{\lambda}_1(u) \sin(u+\phi) - \tilde{\lambda}_2(u) \cos(u+\phi), \\
& \tilde{\lambda}_1(u) = \dfrac{1}{{\rm det}A_{T}(u+\phi)}\left[\displaystyle{ \int_I |\partial_s(u+\phi)|^p \cos{(u+\phi)} \, ds \int_I \cos^2{(u+\phi)}\, ds } \right. \\ 
      & \qquad \qquad \qquad \qquad \left. + \displaystyle{ \int_I |\partial_s(u+\phi)|^p \sin{(u+\phi)} \, ds \int_I \sin{(u+\phi)} \cos{(u+\phi)}\, ds } \right], \\
& \tilde{\lambda}_2(u) = \dfrac{1}{{\rm det}A_{T}(u+\phi)}\left[\displaystyle{ \int_I |\partial_s(u+\phi)|^p \cos{(u+\phi)} \, ds \int_I \cos{(u+\phi)} \sin{(u+\phi)}\, ds } \right. \\ 
      & \qquad \qquad \qquad \qquad \left. + \displaystyle{ \int_I |\partial_s(u+\phi)|^p \sin{(u+\phi)} \, ds \int_I \sin^2{(u+\phi)} \, ds } \right], \\
& u(\cdot,0)= u_0(\cdot).     
\end{cases}
\end{align}
where, for $\theta=u+\phi$, we set
\begin{align}\label{defmatrice}
A_{T}(\theta) := 
\begin{pmatrix}
\int_I \sin^2{\theta} \, ds & -\int_I \sin{\theta} \cos{\theta} \, ds \\
-\int_I \sin{\theta} \cos{\theta} \, ds & \int_I \cos^2{\theta} \, ds
\end{pmatrix}, 
\end{align}
and where the initial data  $u_0 : I \to \mathbb{R}$ satisfies 
\begin{align} \label{eq:1.1}
u_0 \in W^{1,p}_{\rm per}(I) \quad \text{and} \quad \int_I \cos{(u_0+\phi)}\, ds = \int_I \sin{(u_0+\phi)}\, ds = 0.
\end{align}
The significance of this system is that, by interpreting the map $u$ above
as the oscillation of the tangential angle, we are able to construct a gradient
flow of $E_p$ by solving \eqref{eq:P} (see \eqref{EQg1} and the discussion
thereafter).
\begin{dfn} \label{def:WS}
We say that $u : I \times [0, T) \to \mathbb{R}$ is a weak solution of \eqref{eq:P} if the following hold{\rm :}
\begin{enumerate}
\item[{\rm (i)}] $u \in L^\infty(0,T; W^{1,p}_{\rm per}(I)) \cap H^1(0,T;L^2(I))${\rm ;} 
\item[{\rm (ii)}] there exist $\lambda_1$, $\lambda_2 \in L^2(0,T)$ such that $u$ satisfies 
\begin{equation} \label{eq:1.2}
\begin{split}
& \int^T_0 \int_I \left[ \left\{ \partial_t u - \lambda_1 \sin(u+\phi) + \lambda_2 \cos{(u+\phi)} \right\} \varphi \right. \\ 
                & \qquad \qquad \left. + \left( |\partial_s(u+\phi)|^{p-2} \partial_s(u+\phi) \right) \partial_s \varphi \right] \, ds dt = 0 
\end{split}
\end{equation}
for any $\varphi \in L^\infty(0,T; W^{1,p}_{\rm per}(I))${\rm ;} 
\item[{\rm (iii)}] for any $t \in [0, T)$, it holds that 
\begin{align} \label{eq:1.3}
\int_I \cos{(u(\cdot, t)+\phi)}\, ds = \int_I \sin{(u(\cdot,t)+\phi)}\, ds = 0; 
\end{align}
\item[{\rm (iv)}] $u(0) = u_0$ in $W^{1,p}_{\rm per}(I)$. 
\end{enumerate}
\end{dfn}
Our main results are the following:

\begin{thm} \label{Theorem:1.1}
Let $1 <p < \infty$. Assume that $u_0 \in W^{1,p}_{ \rm per}(I)$ satisfies \eqref{eq:1.1}.
Then there exists $T>0$ such that \eqref{eq:P} possesses a weak solution in $[0,T)$.
\end{thm}


In fact using  Lemma~\ref{lem:3.20} below  it even follows that the weak solution of Theorem \ref{Theorem:1.1} is unique for $p\geq 2$   under some additional smallness and regularity assumption on $u_{0}$ when $p>2$.

In the special case where $p=2$ we will also prove the following result.
\begin{thm} \label{Theorem:1.2}
Let $p=2$ and assume that  $u_0 \in W^{2,p}_{\rm per}(I)$ satisfies \eqref{eq:1.1}. 
Then \eqref{eq:P} possesses a unique global-in-time  weak solution (that is we may choose an arbitrarily large $T$ in Definition~\ref{def:WS}).  
\end{thm}
We comment that while monotonicity of the energy holds, convergence
 of a subsequence $u(\cdot,t_j)$ to an equilibrum as
$j\rightarrow\infty$ follows via a standard
argument. However, in order to obtain full convergence, we expect
further regularity is required.
There are two ways this arises: first, smoothness for a standard
exponential decay-style argument (see \cite{Huisken} for the classical
occurance of this, and \cite{ideal} for a recent higher-order flow of curves
appearance), and regularity in a Sobolev space of high order for the more
technical \L ojasiewicz-Simon gradient inequality approach (see
\cite{DPS} for a recent higher-order flow of curves example).

\bigskip

The above theorems essentially imply the (short-time) existence of a weak flow for the $p$-elastic energy of inextensible curves (these are curves for which both length and parametrization speed are constant). The connection between the above formulation \eqref{eq:P} and the considered planar closed inextensible curves $\gamma$ we started with at the beginning of this work is thoroughly discussed in Section~\ref{sec:motivation}: there, derivation, choice of the appropriate function spaces, and the equivalence of \eqref{eq:P} to a gradient flow for $E_{p}$ mentioned earlier, is motivated and explained in detail.

We have stated above that for $p>2$ and a smallness condition on the initial data, we can still show uniqueness of the weak solution.
Note that the smallness condition can be thought of as  an initial curve being close to a (possibly multi-covered) circle. In this case it is reasonable to expect again long time existence of the flow, however different  techniques other than those presented here seem to be necessary to draw this conclusion.
This question 
will be the  subject of  future investigation.

\bigskip

As already mentioned we prove short-time existence via a minimizing movements approach. Interestingly, most often difficulties arise when dealing with the Lagrange-multipliers $\tilde{\lambda}_{j}$.
The structure of the article is as follows: after providing the notation and some preliminary results in Section~\ref{sec:2}, we define our approximation scheme in Section ~\ref{sec:3.1}.  In particular we derive a suitable family of approximating maps and approximating Lagrange multipliers. The control of the latter is a very delicate steps  (see Lemma~\ref{est-lambda}). 
The differentiability in time of the Lagrange multipliers is strictly connected to a control of the discrete velocities: this is investigated in   Lemma~\ref{controlVel} (for $p=2$) and Lemma~\ref{lem:controlVelpBig} (for $p>2$). It is with these Lemmas that the restriction $p \geq2$ comes into play (we need to use Lemma~\ref{Lindquist}). 
In Section~\ref{sec:3.2} we discuss in detail the approximating procedure, in particular  we provide a proof for Theorem~\ref{Theorem:1.1} and show  uniqueness  and some smoothing properties of the weak flow, when $p \geq 2$. Finally in Section~\ref{sec:4} we prove Theorem~\ref{Theorem:1.2}. 

\bigskip

\noindent \textbf{Acknowledgements:} This project has been funded by the Deutsche Forschungsgemeinschaft (DFG, German Research Foundation)- Projektnummer: 404870139.
The first author was partially supported by JSPS KAKENHI Grant numbers 15H02058 and 16H03946. 
The third author gratefully acknowledges support from ARC DP 150100375.

\subsection{Motivation}\label{sec:motivation}
To motivate \eqref{eq:P} and the formulations of our main theorems we collect here a few important ideas.
Consider smooth closed \emph{periodic} curves $\gm=\gm(s)$ of fixed length $L$ parametrized by arc-length $s$ over the domain $\bar{I}=[0,L]$.
Let $T=T(s)=\gamma'(s)$ denote the unit tangent of the curve $\gamma$. It is well known  that   a planar curve is uniquely determined by its tangent indicatrix $T$ (up to rotation and translation).
Recall the formulas $T'= \vec{\kappa}=\vk N$, $N'=-\vk T$, as well as $\theta'(s)= \vk(s)$, where  $T=(\cos \theta, \sin \theta)$.
The map $\theta$ is called the indicatrix of the curve $\gamma$.
Let 
$$ \mathcal{A} :=  C^{\infty}_{\rm per} (\bar{I}, S^{1})  
$$
denote the set of all smooth periodic maps from $[0,L]$ into the unit circle.
For a map $T \in \mathcal{A}$ to describe a closed curve we have to impose the condition $\gamma(L)=\gamma(0)$ which translates into
\begin{align}\label{constraint}
\int_{I} T ds =0.
\end{align}
The $p$-elastic energy \eqref{p-energy} (for $1<p<\infty$) can be written as
\begin{align}
E_{p}(\gamma)= \frac{1}{p} \int_{I}|\vec{\vk}|^{p} ds = \frac{1}{p} \int_{I} |\partial_{s} T|^{p} ds=:F_{p}(T).
\end{align}
Instead of considering an $L^{2}$-gradient flow of fourth order for the energy $E_{p}(\gamma)$, 
we take here a variation of the tangent vector $T$ directly from the perspective of the functional $F_{p}(T)$.
This gives rise to a second order parabolic equation.

More precisely, consider variations of type $T_{\epsilon}= \frac{T+\epsilon \varphi}{|T + \epsilon \varphi |}$ for $\epsilon$ small enough and $\varphi \in C^{\infty}_{\rm per}(I,\R^{2})$.
Then 
$$
\frac{d}{d \epsilon} \Big|_{\epsilon =0} T_{\epsilon}=\varphi - (\varphi \cdot T )T=:\varphi^{\perp}.
$$
Since we want to include the constraint \eqref{constraint} we compute
$$ \frac{d}{d \epsilon} \Big|_{\epsilon =0} \left(  F_{p}(T_{\epsilon}) + \vec{\lambda} \cdot \int_{I} T_{\epsilon} ds \right) =0 , $$
where $\vec{\lambda}=(\lambda_{1},\lambda_{2}) \in \R^{2}$ is a Lagrange multiplier.
A direct computation gives
\begin{align*}
0 &= \int_{I} |\partial_{s} T|^{p-2} \partial_{s} T \cdot \partial_{s} \varphi^{\perp} + \vec{\lambda}\cdot \varphi^{\perp} ds= \int_{I} [- (|\partial_{s} T|^{p-2} \partial_{s}T)_{s}  + \vec{\lambda} ] \cdot \varphi^{\perp} ds  \\
& = \int_{I}[ - \nabla_{s} ( |\partial_{s} T|^{p-2} \partial_{s}T) + (\vec{\lambda} \cdot N) N] \cdot \varphi \, ds,
\end{align*}
where $\nabla_{s} \varphi= \partial_{s} \varphi - (\partial_{s} \varphi \cdot T) T$ and where we have used the periodicity property of the maps.

This motivates studying the \emph{second-oder flow}
\begin{align} \label{sys1}
\partial_{t} T &= \nabla_{s} ( |\partial_{s} T|^{p-2} \partial_{s}T) - (\vec{\lambda} \cdot N) N \qquad \text{ in } I \times (0,t_{*}) \\  \label{sys2}
T(\cdot, 0) &=T_{0}
\end{align}
for some smooth initial data $T_{0}$ satisfying \eqref{constraint}.
Here $\vec{\lambda}=  \vec{\lambda}(t)$ is defined by
\begin{align}\label{defvecL}
\vec{\lambda}=(\lambda_{1}, \lambda_{2}): = \left( \int_{I} |\partial_{s}T|^{p} T ds \right) A_{T}^{-1}
\end{align}
where
\begin{align}
A_{T} = \int_{I} N^{t}\,  N ds
\end{align}
is a $2\times 2$ matrix. Note that if $\det A_{T}=0$ then the Lagrange multiplier is not well defined. By definition we have
$$ \vec{\lambda} \cdot A_{T}=\int_{I} (\vec{\lambda} \cdot N) N ds = \int_{I} |\partial_{s}T|^{p} T ds,
$$
so that, as long as the flow is well defined and as smooth as required, we infer
\begin{align*}
\frac{d}{dt} \int_{I} T ds &= \int_{I} T_{t} ds =  \int_{I} \nabla_{s} ( |\partial_{s} T|^{p-2} \partial_{s}T) - (\vec{\lambda} \cdot N) N  ds \\
&= \int _{I} \partial_{s} ( |\partial_{s} T|^{p-2} \partial_{s}T) ds - \int_{} (\partial_{s}( |\partial_{s} T|^{p-2} \partial_{s}T) \cdot T) T ds - \vec{\lambda} \cdot A_{T} =0
 \end{align*}
due to the periodicity of the map $T$ and since $\partial_{s} T \cdot T =0$ and $\partial_{s}^{2}T \cdot T=- |\partial_{s}T|^{2}$. In other words the constraint \eqref{constraint} is satisfied along the flow.

Note that the energy decreases along the flow. Using the computation above,  the fact that $T_{t}$ is a normal vector field and $\int_{I} T_{t} ds =0 $, we find
\begin{align*}
\frac{d}{dt} F_{p}(T) &= \int_{I} [ - \nabla_{s} ( |\partial_{s} T|^{p-2} \partial_{s}T)] \cdot T_{t} ds = \int_{I} [ - \nabla_{s} ( |\partial_{s} T|^{p-2} \partial_{s}T)] \cdot T_{t} ds + \vec{\lambda} \cdot \int_{I} T_{t} ds \\
& = \int_{I} [ - \nabla_{s} ( |\partial_{s} T|^{p-2} \partial_{s}T)] \cdot T_{t}  + (\vec{\lambda} \cdot N) N \cdot T_{t} ds\\ 
&= - \int_{I} |\partial_{t} T|^{2 } ds \leq 0.
\end{align*}

The system \eqref{sys1}, \eqref{sys2} can be converted into a scalar PDE for $\theta: I \times (0, t_{*})$, where 
\begin{align}\label{defiT}
T(s,t)= (\cos \theta(s,t), \sin \theta(s,t)).
\end{align}
Since
\begin{align*} 
 \nabla_{s} ( |\partial_{s} T|^{p-2} \partial_{s}T)&=  \partial_{s} ( |\partial_{s} T|^{p-2}) \partial_{s}T  +
 |\partial_{s} T|^{p-2} ( \partial_{s}^{2} T - (\partial_{s}^{2} T \cdot T) T)\\
 &=(|\theta_{s}|^{p-2})_{s} \theta_{s}N + |\theta_{s}|^{p-2} \theta_{ss} N,
\end{align*}
we infer (for some fixed $\eta \in \N$ that essentially fixes the rotation index of the curve)
\begin{align}\label{systheta1}
\theta_{t}&= (|\theta_{s}|^{p-2} \theta_{s})_{s} + \lambda_{1}(t) \sin \theta - \lambda_{2}(t) \cos \theta  \qquad \text{ in } I \times (0, t_{*})\\
\label{systheta2gen}
\theta(L, t) - \theta(0, t)& =2\pi \eta \qquad \text{ for all } t,\\
\label{systheta2bis}
|\theta_{s}|^{p-2} \theta_{s} (0, t) & =  |\theta_{s}|^{p-2} \theta_{s} (L, t) \qquad \text{ for all } t,\\
\label{systheta3}
\theta( \cdot, 0) &= \theta_{0}(\cdot).
\end{align}
Condition \eqref{systheta2bis} replaces/stresses  the periodicity property of the map $T$ that has been used several times in the preceding calculations.

The Lagrange multipliers are computed using \eqref{defvecL}. More precisely, since $A_{T}(\theta)$  is given as in \eqref{defmatrice}, 
we get
\begin{align*}
\lambda_{1}(t)&=\frac{1}{\det A_{T}(\theta)} \left(  \int_{I} |\theta_{s}|^{p} \cos \theta \, ds  \int_{I} \cos^{2} \theta \, ds   +  \int_{I} |\theta_{s}|^{p} \sin \theta \, ds  \int_{I} \cos \theta  \sin \theta\, ds  \right), \\
\lambda_{2}(t)&=\frac{1}{\det A_{T}(\theta)} \left(  \int_{I} |\theta_{s}|^{p} \cos \theta \, ds  \int_{I} \cos \theta \sin \theta  \, ds   +  \int_{I} |\theta_{s}|^{p} \sin \theta \, ds  \int_{I}   \sin^{2} \theta\, ds  \right) .
\end{align*}

Note that by Cauchy-Schwarz we have $\det A_{T} \geq 0$. 
For the Lagrange multiplier to be well-defined we need a bound from below on the determinant: this is shown below (see Lemma~\ref{lemLM}). 

Although the next result is not surprising, we report it here in detail since many of the calculations therein will be important for our later discussion.
\begin{lem}\label{energieabnahme}
Suppose $\theta$ is a sufficiently smooth solution of \eqref{systheta1},\eqref{systheta2gen},\eqref{systheta2bis}, \eqref{systheta3}. Then
$$F_{p}(\theta):= \frac{1}{p} \int_{I} |\theta_{s}|^{p} ds$$ decreases along the flow, i.e. $\frac{d}{dt} F_{p}(\theta) \leq 0$.
\end{lem}
\begin{proof} 
We have
\begin{align*}
\frac{d}{dt} F_{p}(\theta) &= \int_{I} |\theta_{s}|^{p-2} \theta_{s} \theta_{st} \,ds=[|\theta_{s}|^{p-2} \theta_{s} \theta_{t} ]_{0}^{L}
- \int_{I}  (|\theta_{s}|^{p-2} \theta_{s} )_{s}\theta_{t} \,ds= -\int_{I}  (|\theta_{s}|^{p-2} \theta_{s} )_{s}\theta_{t}\, ds
\end{align*}
with the boundary disappearing thanks to \eqref{systheta2gen}, \eqref{systheta2bis}. 
Writing $\vec{\lambda} = (\lambda_{1}(t), \lambda_{2}(t))$, $T(\theta)=T(s,t)= (\cos \theta(s,t), \sin \theta (s,t))$, $N(\theta)=
N(s,t)= (-\sin \theta(s,t), \cos \theta(s,t))$ , equation \eqref{systheta1} reads
\begin{align}\label{systheta1concise}
\theta_{t}= (|\theta_{s}|^{p-2} \theta_{s} )_{s} - ( \vec{\lambda} \cdot N(\theta) ).
\end{align}
It follows that
\begin{align*}
\frac{d}{dt} F_{p}(\theta) = -\int_{I} |\theta_{t}|^{2}  ds- (\vec{\lambda} \cdot \frac{d}{dt} \int_{I} T(\theta) ds ) = -\int_{I} |\theta_{t}|^{2}  ds \leq 0
\end{align*}
if and only if  $ \frac{d}{dt} \int_{I} T(\theta) ds  = (0,0)$.
Next we show that the definition of  $\vec{\lambda}$ gives $ \frac{d}{dt} \int_{I} T(\theta) ds  = (0,0)$.
By definition we have that
$$ (\lambda_{1}(t), \lambda_{2}(t)) \cdot A_{T}(\theta) = \int_{I} |\theta_{s}|^{p} (\cos \theta, \sin \theta) ds,$$ which can be equivalently written as
\begin{align}\label{vorrei}
[-\int_{I} (\vec{\lambda} \cdot  N(\theta)) \sin \theta ds,  \int_{I} (\vec{\lambda} \cdot  N(\theta)) \cos \theta ds] = \int_{I}  (\vec{\lambda} \cdot  N(\theta)) N(\theta) ds= \int_{I} |\theta_{s}|^{p} (\cos \theta, \sin \theta) ds.
\end{align}
Again this can be expressed as
\begin{align}\label{lagmult}
 \int_{I}  (\vec{\lambda} \cdot  N(\theta)) N(\theta) ds&= \int_{I} |\theta_{s}|^{p} T(\theta) ds=  \int_{I} |\theta_{s}|^{p-2} \theta_{s} \theta_{s} T(\theta) ds = - \int_{I} |\theta_{s}|^{p-2} \theta_{s} (N(\theta))_{s}  ds
 \notag \\
 & = -[|\theta_{s}|^{p-2} \theta_{s} N(\theta)]_{0}^{L} + \int_{I} (|\theta_{s}|^{p-2} \theta_{s})_{s} N(\theta) ds.
\end{align}
Therefore we infer by \eqref{systheta1concise} and the above expression that
\begin{align}\label{derivataT}
 \frac{d}{dt} \int_{I} T(\theta) ds &= \int_{I} \theta_{t} N(\theta) ds = \int_{I} ( (|\theta_{s}|^{p-2} \theta_{s} )_{s} - ( \vec{\lambda} \cdot N(\theta) )) N(\theta) = [|\theta_{s}|^{p-2} \theta_{s} N(\theta)]_{0}^{L}\\
 & =[|\theta_{s}|^{p-2} \theta_{s} (-\sin \theta, \cos \theta)]_{0}^{L} = (-\sin \theta(0), \cos \theta(0)) [|\theta_{s}|^{p-2} \theta_{s} ]_{0}^{L} \nonumber
 \end{align}
 where we have used that $\sin \theta(0,t)=\sin \theta(L, t)$, $\cos \theta(0,t)=\cos \theta(L, t)$ due to \eqref{systheta2gen}.
 The claim now follows from \eqref{systheta2bis}.
\end{proof}

Note that there is a natural bound from below for the energy.
\begin{lem}\label{fenchel}
We have that  
$$ F_{p}(\theta) \geq \frac{1}{p} \frac{(2\pi \eta)^{p}}{L^{p-1}}.$$
\end{lem}
\begin{proof}
We have that $
2\pi \eta = \int_{0}^{L} \theta_{s} \leq \| \theta_{s}\|_{L^{p}(I)} L^{1-\frac{1}{p}}= (p F_{p}(\theta))^{\frac{1}{p}} L^{1-\frac{1}{p}}
$ 
which gives the claim. 
\end{proof}


\begin{rem}\label{rem1}
Assume that $\theta$ is a (sufficiently smooth) solution of \eqref{systheta1}, where $\lambda_{1}(t)$ and $\lambda_{2}(t)$ are well defined and computed as indicated above.
Set $\tilde{\theta}(s,t):= \theta(s,t) + \alpha(t)$ for some $\alpha(t) \in \R$ and compute $\tilde{\lambda}_{1}(t)$, $\tilde{\lambda}_{2}(t)$
accordingly. Then $\det A_{T}(\theta) = \det A_{T}(\tilde{\theta})$ follows from \eqref{eq:2.2}.
Moreover a long but straightforward computation gives  that
$$ \lambda_{1}(t) \sin \theta - \lambda_{2}(t) \cos \theta = \tilde{\lambda}_{1}(t) \sin \tilde{\theta} - \tilde{\lambda}_{2}(t) \cos \tilde{\theta}.$$
Obviously $(|\theta_{s}|^{p-2} \theta_{s})_{s} = (|\tilde{\theta}_{s}|^{p-2} \tilde{\theta}_{s})_{s}$.
Finally, if $\alpha$ does not depend on time also $\theta_{t}=\tilde{\theta}_{t}$.
Hence we see that $\tilde{\theta}:=\theta + const.$ also solves the PDE \eqref{systheta1}. This is in accordance with the geometrical fact that a rigid rotation of the curve $\gamma$ (whose indicatrix is $\theta$) does not change its geometry and so must be again a solution.
\end{rem}

In view of \eqref{systheta2gen} and  Remark~\ref{rem1} it is reasonable to write $\theta$ as
\begin{align}\label{EQg1}
\theta(s,t)=u(s,t) + \phi(s), \qquad \phi(s) := \frac{2\pi \eta }{L}s
\end{align}
with $u(\cdot,t)$ a periodic map i.e. $u(0,t)=u(L,t)$. Although at this stage it could make sense to additionally  ask that  $\int_{I} u ds =0$ (in order to factor out the translation invariance for $\theta$ mentioned above), it turns out that this constraint is detrimental for the discretization procedure we will use later on. This  and Lemma~\ref{lemLM} below motivates the choice of maps  $u \in W^{1,p}_{\rm per}(I)$.
Writing $u(s,t)=\theta(s,t)-\phi(s)$   the system \eqref{systheta1}, \eqref{systheta2gen}, \eqref{systheta3} is readily transformed into \eqref{eq:P}.
One notices immediately that we have  omitted the equation corresponding to \eqref{systheta2bis}. Since we will be working with $W^{1,p}_{\rm per}(I)$-maps such a condition at the boundary does not make any sense at first. However we will show that the problem \eqref{eq:P} admits a solution and that this solution is sufficiently smooth and satisfies  $|(u+ \phi)_{s}|^{p-2} (u+ \phi)_{s} (0, t)= |(u+ \phi)_{s}|^{p-2} (u+ \phi)_{s}(L,t)$ for all times as desired (basically it turns out to be a natural boundary condition).

\section{Notation and useful preliminary estimates}\label{sec:2}
In the following  we set  for $I=(0,L)$
\begin{align*}
C_{\rm per}(I) &:= \{ u \in C(\mathbb{R}) \mid u(s+kL)=u(s) \quad \text{for any} \quad s \in \bar{I} \quad \text{and} \quad k \in \mathbb{Z} \}, \\
W^{1,p}_{\rm per}(I) &:= \{ u \in W^{1,p}(I) \mid u \in C_{\rm per}(I) \},\\
W^{2,p}_{\rm per}(I) &:= \{ u \in W^{2,p}(I) \mid u, \, \partial_{s}u \in C_{\rm per}(I) \}. 
\end{align*}
Note that since $W^{1,p}(a,b)$ is embedded in the space of continous functions $C([a,b])$ for any $-\infty <a<b < \infty$ (see for instance \cite[Thm.2.2]{BGH}) the definitions are meaningful.
Also  for $u \in W^{1,p}_{\rm per}$ we write $u+\phi=\theta$ where $\phi(s)=2\pi\eta \frac{s}{L}$.

First of all we show that the  determinants considered in this work  (recall \eqref{eq:P}) never degenerate.
\begin{lem} \label{lemLM}
Let $u \in W^{1,p}_{\rm per}(I)$. 
Then it holds that 
\begin{align*}
{\rm det}A_{T}(u+\phi) \ge \dfrac{1}{4} \delta^2, 
\end{align*}
where 
\begin{align} \label{eq:2.1}
\delta:= \min \left\{ \dfrac{L}{2}, \left(\dfrac{\pi}{8 \| \partial_s u + \partial_s \phi \|_{L^p(I)}}\right)^{\tfrac{p}{p-1}} \right\}. 
\end{align}
\end{lem}
\begin{proof}
The proof is a direct modification of the proof of Lemma 2.1 in \cite{Wen}. 
As usual let $\theta = u + \phi$. 
By a direct calculation, we see that 
\begin{align} \label{eq:2.2}
{\rm det}A_{T}(u+\phi) = \dfrac{1}{2} \iint_{I \times I} \sin^2{(\theta(\sigma) - \theta(s))} \, d\sigma ds. 
\end{align}
Without loss of generality we may assume that $\theta(0)=0$ and $\theta(L)=2\pi \eta$ (if not consider the map $\theta -\theta(0)$).
Let $s_0 \in (0, L)$ be such that $\theta(s_0)=\pi/2$ and  
let $E := [-\delta + s_0, s_0 + \delta] \times [-\delta, \delta]$, where 
\begin{align*}
\delta:= \min \left\{ \dfrac{L}{2}, \left(\dfrac{\pi}{8 \| \partial_s \theta \|_{L^p(I)}}\right)^{\tfrac{p}{p-1}} \right\}. 
\end{align*}
Then   since
\begin{align*}
|(\theta(s_2) - \theta(\tilde{s}_2)) - (\theta(s_1) - \theta(\tilde{s}_1))| 
& \le | \theta(s_2) - \theta(s_{1}) | + | \theta(\tilde{s}_2) - \theta(\tilde{s}_1) | \\
& \le  \| \partial_s \theta \|_{L^p(I)} \left( | s_2 - s_{1} |^{1-\tfrac{1}{p}} + | \tilde{s}_2- \tilde{s}_1 |^{1-\tfrac{1}{p}} \right),  
\end{align*}
we deduce that, for any $(s, \tilde{s}) \in E$, 
\begin{align*} 
| (\theta(s)-\theta(\tilde{s})) - (\theta(s_0)-\theta(0)) | 
 \le 2 \| \partial_s \theta \|_{L^p(I)} \delta^{1-\tfrac{1}{p}}  
 \le \dfrac{\pi}{4},  
\end{align*}
that is, 
\begin{align} \label{eq:2.3}
\dfrac{\pi}{4} \le | \theta(s)-\theta(\tilde{s})| \le \dfrac{3}{4} \pi. 
\end{align}
Therefore, we observe from \eqref{eq:2.2} and \eqref{eq:2.3} that 
\begin{align*}
{\rm det}A_{T}(u+\phi) 
 & \ge \dfrac{1}{2} \iint_{E} \sin^2{(\theta(\sigma) - \theta(s))} \, d\sigma ds 
  \ge \dfrac{1}{2} \iint_{E} \sin^2{\dfrac{\pi}{4}} \, d\sigma ds 
  = \dfrac{1}{4} \delta^2. 
\end{align*}
\end{proof}

By Lemma \ref{lemLM} we have the following estimate on $\tilde{\lambda}_1$ and $\tilde{\lambda}_2$. 
\begin{lem} \label{lem4.2}
Let $u \in W^{1,p}_{\rm per}(I)$ and $\tilde{\lambda}_{i}$ be defined as in  \eqref{eq:P}. Then it holds that 
\begin{align*}
|\tilde{\lambda}_i(u)| \le 8 L \| \partial_s (u + \phi) \|_{L^p(I)}^p \left( \dfrac{4}{L^2} + \left(\dfrac{8 \| \partial_s (u + \phi) \|_{L^p(I)}}{\pi} \right)^{\tfrac{2p}{p-1}} \right)   
\end{align*} 
for $i=1,2$. 
\end{lem}
\begin{proof}
Thanks to Lemma \ref{lemLM}, it is suffices to estimate $\delta^{-2}$. 
Since $\max\{ a, b\} = \{ |a+b| + |a-b| \}/2$ for $a$, $b \ge 0$, we have 
\begin{align*}
\delta^{-2} 
 &\le \dfrac{1}{4}\left( \left| \dfrac{2}{L} + \left(\dfrac{8 \| \partial_s (u + \phi) \|_{L^p(I)}}{\pi} \right)^{\tfrac{p}{p-1}} \right| 
                                  +  \left| \dfrac{2}{L} - \left(\dfrac{8 \| \partial_s (u + \phi) \|_{L^p(I)}}{\pi} \right)^{\tfrac{p}{p-1}} \right| \right)^2 \\
 &\le \dfrac{4}{L^2} + \left(\dfrac{8 \| \partial_s (u + \phi) \|_{L^p(I)}}{\pi} \right)^{\tfrac{2p}{p-1}}. 
\end{align*}
Recalling the definition of $\tilde{\lambda}_{i}$ we obtain the conclusion.
\end{proof}


Next we give a few lemmata that will prove useful for deriving  estimates.
\begin{lem} \label{Lem:1.1}
Let $p > 1$. Then there exists $c_{p}>0$ such that
\begin{align*}
\dfrac{1}{c_{p}} ( x^{p-1} + y^{p-1}) \le \dfrac{| x^{p} - y^{p} |}{|x-y|} \le c_{p}(x^{p-1} + y^{p-1}) \quad \text{for any} \quad x, y \ge 0. 
\end{align*} 
\end{lem}
\begin{proof} See \cite[Lemma 3.14]{FFLM}. 
\end{proof}
\begin{rem}\label{remOkabe}
From Lemma \ref{Lem:1.1} and H\"older's inequality, we obtain  for maps $h,g \in L^{p}(I)$ that 
\begin{align*}
\int_{I} \left| |h(s)|^{p} - |g(s)|^{p}\right| \, ds 
 & \le c_{p} \int_{I} |h(s) - g(s)| (|h(s)|^{p-1} + |g(s)|^{p-1}) \, ds \\
 & \le c_{p} \| h - g \|_{L^{p}(I)} (\| h \|_{L^{p}(I)}^{p-1} + \| g \|_{L^{p}(I)}^{p-1} ). 
\end{align*}
\end{rem}
\begin{rem}\label{remOkabe2}
From Lemma \ref{Lem:1.1} and H\"older's inequality, we obtain  for maps $h,g \in L^{p}(I) \cap L^{\infty}(I)$ that 
\begin{align*}
\int_{I} \left| |h(s)|^{p} - |g(s)|^{p}\right| \, ds 
 & \le c_{p} \int_{I} |h(s) - g(s)| (|h(s)|^{p-1} + |g(s)|^{p-1}) \, ds \\
 & \le c_{p}  \| h - g \|_{L^{1}(I)} (\| h \|_{L^{\infty}(I)}^{p-1} + \| g \|_{L^{\infty}(I)}^{p-1} )\\
 & \leq c_{p} \sqrt{L}  \| h - g \|_{L^{2}(I)} (\| h \|_{L^{\infty}(I)}^{p-1} + \| g \|_{L^{\infty}(I)}^{p-1} ). 
\end{align*}
\end{rem}
\begin{lem}\label{lemDB}
Let $p \geq 2$. Then for any $a$, $b \in \R^{m}$, $m \in \N$
\begin{align}
\langle |a|^{p-2} a- |b|^{p-2}b, a-b \rangle \geq C |a-b|^{p}
\end{align}
where $C$ depends only upon $m$ and $p$.
\end{lem}
\begin{proof}
See for instance \cite[I, Lemma~4.4]{DiBenedetto}.
\end{proof}
\begin{lem}\label{Lindquist}
Let $a$, $b \in \R^{m}$ and $p \geq 2$. Then
\begin{align*}
\langle |a|^{p-2} a- |b|^{p-2}b, a-b \rangle &\geq \frac{1}{2} (|b|^{p-2}+ |a|^{p-2}) |a-b|^{2},\\
\Big| |a|^{p-2} a- |b|^{p-2}b \Big| &\leq (p-1) |a-b| \int_{0}^{1} |a+t (b-a)|^{p-2} dt.
\end{align*}
\end{lem}
\begin{proof}
See  the beautiful notes \cite{Lindqvist} by P. Lindqvist. 
\end{proof}
\begin{lem}\label{Hilfsatz1}
Let $u \in W^{1,p}(I) $ for some $1\leq p < \infty$. Then
\begin{align*}
\|u \|_{L^{p}(I)} \leq C ( \| u_{s} \|_{L^{p}(I)} + \| u \|_{L^{2}}),
\end{align*}
where $C=C(I,p)$.
\end{lem}
\begin{proof}
By embedding theory $u \in C^{0}(\bar{I})$, therefore in particular $u \in L^{2}(I)$. For the estimate we use Poincare-inequality as follows
\begin{align*}
\| u \|_{L^{p}(I)} \leq \| u -u_{I} \|_{L^{p}(I)} + \|u_{I} \|_{L^{p}(I)} \leq |I| \| u_{s} \|_{L^{p}(I)} + |I|^{\frac{1}{p}-\frac{1}{2}} \| u \|_{L^{2}(I)},
\end{align*}
where $u_{I}:= \frac{1}{|I|} \int_{I} u ds$.
\end{proof}

\section{Short time existence via minimizing movements}
\subsection{Discretization procedure}\label{sec:3.1}

We prove  the existence of local-in-time weak solution of \eqref{eq:P} via minimizing movements.
Let $u_0 \in  W^{1,p}_{\rm per}(I)$ and set $T>0$, $n \in \N$, and $\tau_n = T/n$. 
We define a family of maps $\{ u_{i,n} \}^{n}_{i=0} \subset W^{1,p}_{\rm per}(I)$ inductively by making use of a minimization problem. 
Set $u_{0,n}= u_0$. 
For each $i \in \{1, 2, \ldots, n\}$, we consider the following variational problem: 
\begin{align} \label{Min} \tag{M$_{i,n}$}
\min\{ G_{i,n}(u) \mid u \in  W^{1,p}_{\rm per}(I) \} 
\end{align}
with 
\begin{align*}
G_{i,n}(u) := F_{p}(u) + P_{i,n}(u) + H_{i,n}(u),  
\end{align*}
where 
\begin{align}
\label{FpU}
F_{p}(u):& = F_{p}(\theta) = \frac{1}{p} \int_{I} |u_{s}+ \phi_{s}|^{p} \,ds, \\
P_{i,n}(u) &:= \dfrac{1}{2 \tau_{n}} \int_{I} | u(s) - u_{i-1,n}(s) |^{2}  \, ds,\\
L_{i,n}(u) &:= \tilde{\lambda}_{1}(u_{i-1,n}) \int_{I} \cos (u + \phi) ds + \tilde{\lambda}_{2}(u_{i-1,n}) \int_{I} \sin (u + \phi) ds, \\
H_{i,n}(u) & := L_{i,n}(u) - L_{i,n}(u_{i-1,n}).
\end{align}
Here $\tilde{\lambda}_{i}$, $i=1,2$ are  given in \eqref{eq:P}.
Note that
\begin{align} \label{notuccia} P_{i,n} (u_{i-1,n}) =0=H_{i,n}(u_{i-1,n}). \end{align}
Moreover note that for $\varphi \in W^{1,p}_{\rm per}(I)$ the first variation is given by
\begin{align} \label{firstvar}
&\frac{d}{d\epsilon}\Big|_{\epsilon=0} G_{i,n}(u + \epsilon \varphi) =   \\
&\quad=\int_{I} \frac{(u(s)-u_{i-1,n}(s))}{\tau_{n}} \varphi(s) ds -\tilde{\lambda}_{1}(u_{i-1,n}) \int_{I} \sin (u + \phi) \varphi(s) ds \nonumber \\
&\quad \quad + \tilde{\lambda}_{2}(u_{i-1,n}) \int_{I} \cos (u + \phi) \varphi(s) ds 
  +\int_{I} \left( \left|u_{s} +  2\pi \frac{\eta}{L}\right|^{p-2} (u_{s} +  2\pi \frac{\eta}{L}) \right ) \vp_{s} ds,
\nonumber
\end{align}
with the convention that $|0|^{p-2}0=0$ for any $1<p<\infty$.
\begin{thm}\label{thm:3.1}
Let $1<p<\infty$. The minimization problem \eqref{Min} 
admits a solution $u_{i,n} \in W^{1,p}_{\rm per}(I)$.
\end{thm}
\begin{proof}
Note that by Lemma~\ref{lem4.2} the functional $G_{i,n}$ is well defined and bounded from below (since $|L_{i,n}(u)| \leq C$, with $C=C(\|(u_{i-1,n})_{s}\|_{L^{p}(I)})$).
Let $(u_{j})_{j \in \mathbb{N}} \subset W^{1,p}_{\rm per}(I)$ be a minimizing sequence for the problem \eqref{Min}. Then there exists a contant $C >0$ such that $G_{i,n}(u_{j}) \leq C$ for all $j$. From the boundedness of $F_{p}(u_{j})$ and $P_{i,n}(u)$ and  we infer that $\| (u_{j})_{s} \|_{L^{p}(I)} + \|u_{j} \|_{L^{2}(I)} \leq C$ for all $j$. 
Lemma~\ref{Hilfsatz1} yields that $\| u_{j}\|_{L^{p}(I)} \leq C$. Thus  $\| u_{j} \|_{W^{1,p}(I)} \leq C$ for all $j$.
By embedding theory and standard arguments we infer the existence of   $u \in W^{1,p}_{\rm per}(I) \cap C^{0, \alpha}(I)$ with $\alpha= 1-\frac{1}{p}$ such that (passing to a subsequence) we have
\begin{align} \label{eq:3.5}
u_{j} \rightharpoonup u \quad \text{weakly in} \quad W^{1,p}(I)
\end{align}
and
\begin{align} \label{eq:3.6}
u_{j} \rightrightarrows u \quad \text{ with } u \in   C^{0,\alpha}(I) .
\end{align}
Finally we show that $u$ is a minimizer by using the property of lower semicontinuity of the $L^{p}$-norm,  \eqref{eq:3.5} and \eqref{eq:3.6}, namely 
\begin{align*}
\liminf_{j \to \infty} G_{i,n}(u_{j}) 
  = \liminf_{j \to \infty} \left[ \frac{1}{p} \| \partial_{s} u_{j} +\phi_{s}\|^{p}_{L^{p}(I)} + P_{i,n}(u_{j}) + H_{i,n}(u_{j})\right] 
  \ge G_{i,n}(u). 
\end{align*}
\end{proof}
Note that at this point there is no need to investigate uniqueness of the solution  of the minimization problem above. Indeed one can not obtain the uniqueness of solution to \eqref{eq:P} from the minimizing movements method, even if the solution of \eqref{Min} is unique: this is due to the fact that we pass to subsequences in the approximation procedure.

From now on, we denote by $V_{i,n}$ the discrete velocity, that is  
\begin{align*}
V_{i,n}(s) := \dfrac{u_{i,n}(s) - u_{i-1,n}(s)}{\tau_{n}}. 
\end{align*}

\begin{dfn} \label{definition:2.4}
Let $u_{n} : I \times [\,0, T\,] \to \R$ be defined by 
\begin{align*}
u_{n}(s,t) := u_{i-1,n}(s) + ( t - (i-1) \tau_{n}) V_{i,n}(s)
\end{align*} 
if $(s,t) \in I \times [\,(i-1) \tau_{n}, i \tau_{n}\,]$ for $i=1, \ldots, n$. 
\end{dfn}
Thus $u_{n}$ denotes the piecewise linear interpolation of $\{ u_{i,n} \}$.  It is useful to consider also  piecewise constant versions $\tilde{u}_{n}$ and $\tilde{U}_{n}$ of $\{ u_{i,n} \}$, and this is given in the following definition.

\begin{dfn} \label{definition:2.5}
Let $\tilde{u}_{n} : I \times (\,0, T\,] \to \R$ and $V_{n} : I \times (\,0, T\,] \to \R$ be defined by 
\begin{align*}
\tilde{u}_{n}(s,t) &:= u_{i,n}(s), \\
\tilde{U}_{n}(s,t) & := u_{i-1,n}(s),  \\
V_{n}(s,t) &:= V_{i,n}(s), 
\end{align*} 
if $(s,t) \in I \times (\,(i-1) \tau_{n}, i \tau_{n}\,]$ for $i=1, \ldots, n$. 
\end{dfn}


\subsubsection{Uniform bounds for approximating functions ($1<p<\infty$)}

In the following theorem we derive uniforms bounds for the solutions of \eqref{Min}.

\begin{thm} \label{thm:3.2}
Let $1<p<\infty$.
Given a initial data $u_{0}\in W^{1,p}_{\rm per}(I)$, set 
\begin{align}
c_{*}&:=2 \|(u_{0}+\phi)_{s} \|_{L^{p}}^{p},\\
c_{1}&:=8 L c_{*} \left( \dfrac{4}{L^2} + \left(\dfrac{8 c_{*}}{\pi} \right)^{\tfrac{2p}{p-1}} \right) .
\end{align}
Let $T=T(p,u_{0},L)>0$ be such that
$$T \leq \frac{c_{*}}{8pLc_{1}^{2}} .$$
Let $u_{i,n}$ be the solution of \eqref{Min} obtained by Theorem {\rm \ref{thm:3.1}}. 
Then, for each $n \in \N$, we have 
\begin{gather}
\sup_{1 \le i \le n} \| ( u_{i,n} +\phi)_{s} \|^{p}_{L^{p}(I)} \le p F_{p}(u_{0}) +  pT L 4c_{1}^{2} \leq c_{*} , \label{eq:3.8}\\
\int^{T}_{0} \int_{I} | V_{n}(s,t) |^{2} \, ds dt \le 4 (F_{p}(u_{0}) +T L4c_{1}^{2}) \leq \frac{4}{p} c_{*}, \label{eq:3.7} \\
\sup_{1 \le i \le n} \| u_{i,n} \|_{L^{2}(I)} \leq \|u_{0} \|_{L^{2}(I)}+(4T ( F_{p}(u_{0}) +T L 4 c_{1}^{2}))^{1/2} \leq\|u_{0} \|_{L^{2}(I)}+ \sqrt{\frac{c_{*}^{2}}{2Lp^{2}c_{1}^{2}} } \label{eq:3.8L2},\\
\label{lemconti}
|\tilde{\lambda}_{r}(u_{j-1,n}  )| \leq c_{1}, \qquad r=1,2, \qquad j =1, \ldots, n.
\end{gather}
\end{thm}
\begin{proof}
The proof follows by an induction argument. Fix $i \in \{1, 2, \ldots, n\}$ arbitrarily and assume that $\|  (u_{j,n}+\phi)_{s} \|^p_{L^p(I)} \le c_{*}$ for all $0 \leq j \leq i-1$. 
Then it follows from Lemma \ref{lem4.2} that 
\begin{align} \label{eq:3.11b}
| \tilde{\lambda}_r(u_{j,n})| \le c_1 \quad \text{for} \quad j=1, 2, \quad 0\leq j \leq i-1. 
\end{align}
Next, note that
\begin{align}\label{calcH}
|H_{i,n} (u_{i,n})| &= |L_{i,n} (u_{i,n}) -L_{i,n} (u_{i-1,n})| \\
& \leq |\tilde{\lambda}_{1} (u_{i-1,n})| \int_{I} |\cos (u_{i,n} + \phi) - \cos (u_{i-1,n} + \phi)| ds \notag\\
& \qquad + |\tilde{\lambda}_{2} (u_{i-1,n})| \int_{I} |\sin (u_{i,n} + \phi) - \sin (u_{i-1,n} + \phi)| ds \notag\\
& \leq (|\tilde{\lambda}_{1} (u_{i-1,n})| + |\tilde{\lambda}_{2} (u_{i-1,n})|) \int_{I} |u_{i,n} -u_{i-1,n}| ds \notag\\
& \leq \sqrt{L}(|\tilde{\lambda}_{1} (u_{i-1,n})| + |\tilde{\lambda}_{2} (u_{i-1,n})|) \| u_{i,n} -u_{i-1,n}\|_{L^{2}}\notag\\
& \leq \frac{1}{2} P_{i,n} (u_{i,n}) + \tau_{n} L (|\tilde{\lambda}_{1} (u_{i-1,n})| + |\tilde{\lambda}_{2} (u_{i-1,n})|)^{2}. \notag
\end{align}

Since $u_{i,n}$ is a minimizer of \eqref{Min}, we have by using \eqref{notuccia} that 
\begin{align} \label{eq:3.9}
G_{i,n}(u_{i,n}) \le G_{i,n}(u_{i-1,n}) = F_{p}(u_{i-1,n}). 
\end{align}
This implies that $$ F_{p}(u_{i,n}) + P_{i,n} (u_{i,n})-| H_{i,n}(u_{i,n})|\leq F_{p}(u_{i,n}) + P_{i,n} (u_{i,n})+ H_{i,n}(u_{i,n}) \le F_{p}(u_{i-1,n}) $$ for each $i=1, \ldots, n$, 
so that \eqref{calcH} yields
\begin{align} \label{eq:3.11}
F_p(u_{i,n}) \le F_{p}(u_{i,n}) +\frac{1}{2} P_{i,n} (u_{i,n}) \le F_p(u_{i-1,n}) +\tau_{n} L (|\tilde{\lambda}_{1} (u_{i-1,n})| + |\tilde{\lambda}_{2} (u_{i-1,n})|)^{2} 
\end{align}
for each $i =  1, \ldots, n$.
From \eqref{eq:3.11} and \eqref{eq:3.11b}
we infer that indeed 
\begin{align} 
F_p(u_{i,n})  \le F_p(u_0) +\frac{iT}{n} L 4c_{1}^{2} \quad \text{for each} \quad i =  1, \ldots, n, 
\end{align}
This gives \eqref{eq:3.8} and \eqref{lemconti}.
Next we observe that  \eqref{eq:3.11} and  \eqref{lemconti} give
\begin{align*}\frac{1}{2} P_{i,n}(u_{i,n}) &\leq  F_{p}(u_{i-1,n}) - F_{p}(u_{i,n}) + \tau_{n} L (|\tilde{\lambda}_{1} (u_{i-1,n})| + |\tilde{\lambda}_{2} (u_{i-1,n})|)^{2}\\
& \leq  F_{p}(u_{i-1,n}) - F_{p}(u_{i,n}) + \tau_{n}L 4 c_{1}^{2}.
\end{align*}  
Thus we obtain   
\begin{align} \label{eq:3.10}
\frac{1}{2}P_{i,n}(u_{i,n})=  \dfrac{\tau_{n}}{4} \int_{I} | V_{i,n}(s) |^{2} \, ds 
  \le F_{p}(u_{i-1,n}) - F_{p}(u_{i,n}) + \tau_{n} L 4 c_{1}^{2}.
\end{align}
It follows that 
\begin{align*}
\dfrac{1}{4} \int^{T}_{0} \int_{I} | V_{n}(s,t) |^{2} \, ds dt &
 = \sum^{n}_{i=1}  \dfrac{\tau_{n}}{4} \int_{I} | V_{i,n}(s) |^{2} \, dx \\
 &\le \sum^{n}_{i=1} \left[ F_{p}(u_{i-1,n}) - F_{p}(u_{i,n}) +  \tau_{n} L 4 c_{1}^{2}\right]    \le F_{p}(u_{0}) +T L 4 c_{1}^{2}. 
\end{align*}
Thus we obtain \eqref{eq:3.7}. 
To infer \eqref{eq:3.8L2} we use again \eqref{eq:3.10} as follows
\begin{align*}
\| u_{i,n} \|_{L^{2}(I)} &\leq \sum_{j=1}^{i} \| u_{j,n} -u_{j-1,n} \|_{L^{2}(I)} + \| u_{0,n} \|_{L^{2}(I)}\\
& \leq \sum_{j=1}^{i} \sqrt{2 \tau_{n}} \sqrt{P_{j,n}(u_{j,n})} + \|u_{0} \|_{L^{2}(I)} \leq  (\sum_{j=1}^{i} 2\tau_{n})^{1/2} (\sum_{j=1}^{i} P_{j,n}(u_{j,n}))^{1/2} +\|u_{0} \|_{L^{2}(I)}\\
& \leq \|u_{0} \|_{L^{2}(I)}+\sqrt{4T} (\sum_{j=1}^{i}\left[ F_{p}(u_{j-1,n}) - F_{p}(u_{j,n}) +  \tau_{n} L 4 c_{1}^{2}\right])^{1/2} \\
&\leq\|u_{0} \|_{L^{2}(I)}+ (4T ( F_{p}(u_{0}) +T L 4 c_{1}^{2}))^{1/2} \leq\|u_{0} \|_{L^{2}(I)}+ \sqrt{\frac{4T}{p} c_{*}}.
\end{align*}
\end{proof}


\subsubsection{Regularity of approximating functions ($1<p<\infty$)}\label{sec:regMinimizers}
We now discuss the regularity of the minimizers of \eqref{Min}. 
Recalling \eqref{firstvar} we see that
\begin{align}\label{regeq}
 \int_{I} w \varphi_{s} + \xi \varphi ds =0 \qquad \forall \varphi \in W^{1,p}_{\rm per}(I) 
  \end{align}
holds for
$$w = |(u_{i,n}+\phi)_{s}|^{p-2} (u_{i,n}+\phi)_{s} \in L^{\frac{p}{p-1}}(I)$$ and
$$\xi = \frac{u_{i,n}-u_{i-1,n}}{\tau_{n}} -\tilde{\lambda}_{1}(u_{i-1,n}) \sin (u_{i,n}+\phi) + \tilde{\lambda}_{2}(u_{i-1,n}) \cos (u_{i,n}+\phi) \in W^{1,p}_{\rm per}(I).$$
From equation \eqref{regeq} we infer immediately that $ w$ admits weak derivative and $w_{s}=\xi$, thus $w \in W^{2,p}(I)$.
We claim that actually $w \in W^{2,p}_{\rm per}(I)$.
Indeed testing with $\varphi \in W^{1,p}_{\rm per}(I)$ and integrating by parts we infer
\begin{align*}
0=[w\varphi]_{0}^{L}+ \int_{I} (\xi-w_{s}) \varphi ds = \varphi(0) (w(L)-w(0)).
\end{align*}
Since $\varphi \in W^{1,p}_{\rm per}(I)$ can be chosen arbitrarily it follows that 
\begin{align}\label{periodicityproperty}
w(0)=w(L)
\end{align}
that is $w \in W^{2,p}_{\rm per}(I) $. Finally note that
\begin{align*}
\| w \|_{L^{\infty}} \leq \frac{1}{L} \int_{I} |w| + \int_{I} |w_{s}| ds \leq C( \| (u_{i,n}+\phi)_{s} \|_{L^{p}(I)} 
+  \| V_{i,n} \|_{L^{2}(I)} +  (| \tilde{\lambda}_{1} (u_{i-1,n})| + |\tilde{\lambda}_{2} (u_{i-1,n})|))
\end{align*}
with $C=C(L)$,
so that  using \eqref{lemconti} and \eqref{eq:3.8} we immediatley obtain

\begin{lem}\label{lemma3bis}
Let the assumptions of Theorem~\ref{thm:3.2} hold. Then
\begin{align*}
\| |(u_{i,n}+\phi)_{s}|^{p-2} (u_{i,n}+\phi)_{s}\|_{W^{1,2} (I)} &  \leq C (1 +\| V_{i,n} \|_{L^{2}(I)} ) 
\end{align*}
for all $i=1, \ldots,n$, where $C=C(L, c_{*},p)$. In particular the estimates
\begin{align*}
\| |(u_{i,n}+\phi)_{s}|^{p-2} (u_{i,n}+\phi)_{s}\|_{L^{\infty} (I)} &  \leq C (1 +\| V_{i,n} \|_{L^{2}(I)} ) ,\\
 \| |(u_{i,n}+\phi)_{s}|^{p-2} (u_{i,n}+\phi)_{s}\|_{L^{q} (I)} &  \leq C (1 +\| V_{i,n} \|_{L^{2}(I)} )  ,\\
\|(u_{i,n}+\phi)_{s} \| _{L^{\infty} (I)}^{ p-1} &\leq C (1 +\| V_{i,n} \|_{L^{2}(I)} ) ,\\
 \| |(u_{i,n}+\phi)_{s}|^{p-2} (u_{i,n}+\phi)_{s}\|_{W^{1,1} (I)} &  \leq C (1 +\| V_{i,n} \|_{L^{2}(I)} ) ,
\end{align*}
hold for all $i=1, \ldots,n$, where $C=C(L, c_{*},p)$ and $\frac{1}{p}+\frac{1}{q}=1$.
\end{lem}

\subsubsection{Discrete Lagrange multipliers ($1<p<\infty$)}

Recalling how $\tilde{\lm}_{j}(\cdot)$ is defined in \eqref{eq:P} we now give the following definition:
\begin{dfn}\label{deflambdas}
Let $u_{i,n} \in W^{1,p}_{\rm per}(I)$ denote the solution of \eqref{Min} obtained by Theorem~\ref{thm:3.1}.
We define $\tilde{\lm}_{n}^{j}(t) : (\,0, T\,] \to \R$ for $j=1,2$ by 
\begin{align}
\tilde{\lm}_{n}^{j}(t) = \tilde{\lm}_{j}(u_{i-1,n}) \quad \text{if} \quad t \in (\,(i-1)\tau_{n}, i \tau_{n}\,] \quad \text{for} \quad i = 1, 2, \ldots, n, 
\end{align}
\end{dfn}
Moreover let us define the following  piecewise linear maps:
\begin{dfn}\label{defBiglambdas}
We define $\Lambda_{n}^{1}(t), \Lambda_{n}^{2}(t) : (\,0, T\,] \to \R$ by 
\begin{align}
\Lambda_{n}^{j}(t) = \tilde{\lm}_{j}(u_{i-1,n}) +  (t- (i-1)\tau_{n}) \frac{\tilde{\lm}_{j}(u_{i,n})-\tilde{\lm}_{j}(u_{i-1,n})}{\tau_{n}}\quad \text{if} \quad t \in (\,(i-1)\tau_{n}, i \tau_{n}\,] 
\end{align}
for $i = 1, 2, \ldots, n$ and $j=1,2$.
\end{dfn}


\begin{lem} \label{lem:3.3}
Let the assumptions of Theorem~\ref{thm:3.2} hold. Then
$\tilde{\lm}_{n}^{j} \in L^{2}(0,T)$ for any $n \in \N$, $j=1,2$. In particular, we have 
\begin{align*}
\int^{T}_{0} \tilde{\lm}_{n}^{j}(t)^{2}\, dt \le \frac{c_{*}}{8pL}, \qquad \text{ for } j=1,2. 
\end{align*}
Therefore there exists maps $\lambda_{j} \in L^{2}(0,T)$ towards which  $\tilde{\lm}_{n}^{j} $ converges weakly in $L^{2}(0,T)$.
\end{lem}
\begin{proof}
Using Theorem \ref{thm:3.2} and  \eqref{lemconti} we infer that  
\begin{align*}
\int^{T}_{0} \tilde{\lm}^{j}_{n}(t)^{2}\, dt 
= \sum^{n}_{i=1} \int^{i \tau_{n}}_{(i-1)\tau_{n}} | \tilde{\lm}_{j}(u_{i-1,n})|^{2} \, dt \le \sum^{n}_{i=1}  \tau_{n} c_{1}^{2} \leq Tc_{1}^{2}.
\end{align*}
\end{proof}


For the Lagrange multipliers we  derive further meaningful estimates.
\begin{lem}\label{est-lambda}
Let the assumptions of Theorem~\ref{thm:3.2} hold.
Consider $\tilde{\lm}_{n}^{j}: (0, T] \to \R$ as given in Definition~\ref{deflambdas} and let $t_{1}, t_{2} \in (0,T]$. There exists a  positive constant $C=C(u_{0}, L,p)$ such that for $r=1,2$ we have
\begin{align*}
|\tilde{\lm}_{n}^{r}(t_{2}) -\tilde{\lm}_{n}^{r}(t_{1})| &\leq C \|u_{i-1,n} -u_{j-1,n} \|_{W^{1,p}(I)} \\
&\leq C \|u_{i-1,n} -u_{j-1,n} \|_{L^{2}(I)}+C \|(u_{i-1,n} -u_{j-1,n})_{s} \|_{L^{p}(I)},
\end{align*}
where $i, j \in \{ 1, \ldots, n\}$ are such such that $t_{1} \in (\,(i-1)\tau_{n}, i \tau_{n}\,]$ and $t_{2} \in (\,(j-1)\tau_{n}, j \tau_{n}\,]$. 
If $t_{2}=t+\tau_{n}$ then
\begin{align*}
|\tilde{\lm}_{n}^{r}(t+\tau_{n}) -\tilde{\lm}_{n}^{r}(t)|  \leq C \tau_{n}\|(V_{n}(\cdot,t) )_{s}  \|_{L^{p}(I)} + C \tau_{n}\|V_{n}(\cdot,t)   \|_{L^{2}(I)}.
\end{align*}
If in addition $p \geq 2$ and we know that  $\|(u_{i,n}+\phi)_{s} \|_{L^{\infty}}^{p-1} \leq C (1 + \| V_{i,n} \|_{L^{2}(I)})$ for all $i=0, \ldots, n$ (for some appropriately defined $V_{0,n} \in L^{2}(I)$), then
\begin{align}\label{est-lambda-forte}
| \tilde{\lambda}_{r}(u_{i,n}) -\tilde{\lambda}_{r}(u_{i-1,n})| &\leq C \tau_{n} \| (V_{i,n})_{s} \|_{L^{2}(I)} (1+  \| V_{i,n} \|_{L^{2}(I)} 
+  \| V_{i-1,n} \|_{L^{2}(I)}) \\
& \quad + C \tau_{n} \| V_{i,n} \|_{L^{2}(I)}   \qquad \text{ for }\qquad r=1,2 \quad i=1, \ldots,n . \notag
\end{align}
\end{lem}
Note that Lemma~\ref{lemma3bis} yields exactly the additional condition required to infer \eqref{est-lambda-forte}.
\begin{proof}
Let $t_{1} \neq t_{2}$. There exist $i, j \in \{ 1, \ldots, n\}$ such that $t_{1} \in (\,(i-1)\tau_{n}, i \tau_{n}\,]$ and $t_{2} \in (\,(j-1)\tau_{n}, j \tau_{n}\,]$. 
Then by definition we have 
$$  \tilde{\lm}_{n}^{r}(t_{2}) -\tilde{\lm}_{n}^{r}(t_{1})= \tilde{\lm}_{r}(u_{j-1,n}) - \tilde{\lm}_{r}(u_{i-1,n}), \qquad  r \in \{ 1,2 \} .$$
If $i=j$ then $\tilde{\lm}_{n}^{r}(t_{2}) =\tilde{\lm}_{n}^{r}(t_{1})$, therefore let us assume that $i\neq j$.
Then  we have that
$ 0<|t_{2} -t_{1}| \leq (|j-i|+1) \tau_{n}$.
To make the reading easier let us set $v:= u_{i-1,n}$ and $w:=u_{j-1,n}$. Also let us choose $r=1$ (the case $r=2$ is treated exactly in the same way). We have that $u,v \in W^{1,p}_{\rm per}(I)$. Further, by Theorem~\ref{thm:3.2} and Lemma~\ref{Hilfsatz1}  we we know that $$\|v \|_{W^{1,p}(I)}, \|w \|_{W^{1,p}(I)} \leq C=C(u_{0}, L,p,\phi).$$ 
By definition (recall \eqref{eq:P}) we have that
\begin{align*}
&\tilde{\lambda}_{1}(v) - \tilde{\lambda}_{1}(w)=\frac{1}{\det A_{T}(v+\phi)} \Big(  \int_{I} |v_{s} + 2\pi \frac{\eta}{L}|^{p} \cos (v+\phi) \, ds  \int_{I} \cos^{2} (v+\phi) \, ds    \\
& \qquad   \qquad \qquad  \qquad \qquad  \qquad \qquad+  \int_{I} |v_{s} +2\pi \frac{\eta}{L}|^{p} \sin (v+\phi) \, ds  \int_{I} \cos (v+\phi)  \sin (v+\phi)\, ds  \Big), \\
& \quad - \frac{1}{\det A_{T}(w+\phi)} \Big(  \int_{I} |w_{s} + 2\pi \frac{\eta}{L}|^{p} \cos (w+\phi) \, ds  \int_{I} \cos^{2} (w+\phi) \, ds    \\
& \qquad \qquad  \qquad  \qquad \qquad+  \int_{I} |w_{s} +2\pi \frac{\eta}{L}|^{p} \sin (w+\phi) \, ds  \int_{I} \cos (w+\phi)  \sin (w+\phi)\, ds  \Big), \\
&= \left( \frac{1}{\det A_{T}(v+\phi)} - \frac{1}{\det A_{T}(w+\phi)} \right ) \Big(  \int_{I} |v_{s} + 2\pi \frac{\eta}{L}|^{p} \cos (v+\phi) \, ds  \int_{I} \cos^{2} (v+\phi) \, ds  +\ldots \Big) \\
& \quad +  \frac{1}{\det A_{T}(w+\phi)}  \Big(  \int_{I} |v_{s} + 2\pi \frac{\eta}{L}|^{p} \cos (v+\phi) \, ds  \int_{I} \cos^{2} (v+\phi) \, ds  +\ldots \\
& \quad \qquad \qquad \qquad \qquad   -  \int_{I} |w_{s} + 2\pi \frac{\eta}{L}|^{p} \cos (w+\phi) \, ds  \int_{I} \cos^{2} (w+\phi) \, ds  - \ldots \Big ) .
\end{align*}
Using Lemma~\ref{lemLM}, the above bounds for $v$ and $w$,   the mean value theorem, embedding theory (that is employing $\|v-w\|_{L^{\infty}} \leq C \| u-w \|_{W^{1,p}(I)}$) and 
 Remark~\ref{remOkabe} to evaluate differences of type
 $ \int_{I} \left||v_{s} + 2\pi \frac{\eta}{L}|^{p}- |w_{s} + 2\pi \frac{\eta}{L}|^{p} \right | \, ds$  a lengthy but straightforward calculation gives
 \begin{equation}\label{LIPlambda}
  |\tilde{\lambda}_{1}(v) - \tilde{\lambda}_{1}(w)| \leq C \| v-w \|_{W^{1,p}(I)} \leq C\|( v-w)_{s} \|_{L^{p}(I)} + C\| v-w \|_{L^{2}(I)},
  \end{equation}
  where we have used Lemma \ref{Hilfsatz1} in the second inequality.
Thus we can write
 \begin{align*}
 |\tilde{\lm}_{n}^{r}(t_{2}) -\tilde{\lm}_{n}^{r}(t_{1})| &\leq C \|u_{i-1,n} -u_{j-1,n} \|_{W^{1,p}(I)} \\
 & \leq C \|(u_{i-1,n} -u_{j-1,n})_{s} \|_{L^{p}(I)} + C\|u_{i-1,n} -u_{j-1,n} \|_{L^{2}(I)}  .
 \end{align*}
 The second claim follows by taking $t_{1}=t$, $t_{2}=t+\tau_{n}$, (thus $j=i+1$) and the definition of the velocity $V_{n}$.   
 The third claim follows by arguing as above and by being a bit more careful in the estimate.
 More precisely usage of the embedding theorem must be avoided:  this is done as follows. Employing Lemma~\ref{lemLM}, the uniform  bounds for $v$ and $w$, $p \geq2 $, and  the mean value theorem, one evaluates as follows terms such as
\begin{align*}
 &|\int_{I} |(v + \phi)_{s}|^{p} (\cos (v+\phi) -\cos (w +\phi) ) \, ds| 
 \leq C \|(v + \phi)_{s} \|_{L^{\infty}}^{p-1}  \int_{I} |(v+\phi)_{s}||v-w| ds  \\
 &\qquad \leq C \|(v + \phi)_{s} \|_{L^{\infty}}^{p-1}\| (v+\phi)_{s} \|_{L^{2}(I)} \|( v-w)_{s} \|_{L^{2}(I)} \leq  C \|(v + \phi)_{s} \|_{L^{\infty}}^{p-1} \|( v-w)_{s} \|_{L^{2}(I)} ,\\
 & \int_{I} |\cos ^{2}(v+\phi) - \cos ^{2}(v+\phi)| ds  \leq C \int_{I} |v-w| ds \leq C \| v-w \|_{L^{2}(I)}.
 \end{align*}
We additionally employ now Remark~\ref{remOkabe2} and the extra assumption to estimate   differences of type
 $ \int_{I} \left||(u_{i,n} +\phi)_{s} |^{p}- |(u_{i-1,n} +\phi)_{s} |^{p} \right | \, ds$: namely
 \begin{align*}
 \int_{I} \left||(u_{i,n} +\phi)_{s} |^{p}- |(u_{i-1,n} +\phi)_{s} |^{p} \right | \, ds &\leq C \| (u_{i,n} -u_{i-1,n})_{s} \|_{L^{2}(I)}
  (1 + \| V_{i,n} \|_{L^{2}(I)} + \| V_{i-1,n} \|_{L^{2}(I)}) \\
  &= C \tau_{n} \| (V_{i,n})_{s} \|_{L^{2}(I)}(1 + \| V_{i,n} \|_{L^{2}(I)} + \| V_{i-1,n} \|_{L^{2}(I)}) .
 \end{align*}
 \end{proof}
\begin{lem}\label{lem:Lambdazero}
Let the assumptions and notation of Theorem~\ref{thm:3.2} hold.  Then  for $\Lambda^{j}_{n}$, $j=1,2$ (recall Definition  \ref{defBiglambdas}), we have
\begin{align}
&|(\Lambda_{n}^{j})_{t}(t)| \leq C \| V_{n}( \cdot, t)\|_{W^{1,p}(I)}  \qquad \text{for a.e.} \quad  t \in [0,T], \\
& |\Lambda_{n}^{j}(t)| \leq C \qquad \text{for all } t \in [0,T],
\end{align}
where $C=C(p,c_{*},L, u_{0})$.
\end{lem}
\begin{proof}
Following the definition of $\Lambda_{n}^{j}$ and Lemma~\ref{est-lambda} it follows that for  $t \in ((i-1)\tau_{n}, i \tau_{n})$
\begin{align*}
|(\Lambda_{n}^{j})_{t}(t)|= \left|   \frac{\tilde{\lm}_{j}(u_{i,n})-\tilde{\lm}_{j}(u_{i-1,n})}{\tau_{n}}  \right| \leq C\| V_{n}( \cdot, t)\|_{W^{1,p}(I)}.
\end{align*}
The uniform bound  for $|\Lambda_{n}^{j}(t)|$ follows from  
  uniform bounds for
$ |\tilde{\lambda}_{1} (u_{j,n})|$ and 
$|\tilde{\lambda}_{2} (u_{j,n})|$  for $j \in \{1, \ldots,n\}$ by   \eqref{lemconti} in Theorem~\ref{thm:3.2}. 
\end{proof}

Next we would like to understand how the discrete Lagrange multiplier approximate \eqref{vorrei}.

 \begin{lem} \label{lem:211} Assume that Theorem~\ref{thm:3.2} holds. Then 
 for $t \in (0, T]$ we have
 \begin{align}\label{eqRuleL}
 & \int_{I} \langle  \left (\begin{array}{c} \tilde{\lambda}_{n}^{1}(t)\\ \tilde{\lambda}_{n}^{2}(t) \end{array} \right), 
 \left (\begin{array}{c} -\sin (\tilde{U}_{n}(s,t) +\phi)\\ \cos (\tilde{U}_{n}(s,t) +\phi) \end{array} \right)\rangle   \left (\begin{array}{c} -\sin (\tilde{U}_{n}(s,t) +\phi)\\ \cos (\tilde{U}_{n}(s,t) +\phi) \end{array} \right)ds \\
 &= 
 \int_{I} |(\tilde{U}_{n}(s,t)+\phi)_{s}|^{p}  \left (\begin{array}{c} \cos (\tilde{U}_{n}(s,t) +\phi)\\ \sin (\tilde{U}_{n}(s,t) +\phi) \end{array} \right)ds. \notag
 \end{align}
 \end{lem}
 \begin{proof} By definition of  $\tilde{\lambda}_{j}(\cdot)$, $j=1,2$, 
 we have 
 \begin{align*}
 \int_{I} \langle \left (\begin{array}{c} \tilde{\lambda}_{1}(u_{i,n})\\ \tilde{\lambda}_{2}(u_{i,n}) \end{array} \right), &
 \left (\begin{array}{c} -\sin (u_{i,n} +\phi)\\ \cos (u_{i,n} +\phi) \end{array} \right)\rangle   \left (\begin{array}{c} -\sin (u_{i,n} +\phi)\\ \cos (u_{i,n} +\phi) \end{array} \right)ds \\
 &= 
 \int_{I} |(u_{i,n}+\phi)_{s}|^{p}  \left (\begin{array}{c} \cos (u_{i,n} +\phi)\\ \sin (u_{i,n} +\phi) \end{array} \right)ds. 
 \end{align*}
 The claim now follows using the definition of $\tilde{\lambda}^{j}_{n}$, $j=1,2$, and $\tilde{U}_{n}$.
 \end{proof}


\subsubsection{The case $p=2$: Control of the velocities}

Lemma~\ref{est-lambda} and Lemma~\ref{lem:Lambdazero} indicate that the differentiability in time of the Langrange multipliers is strictly connected to a control of the $W^{1,p}$-norm of the discrete velocities. Therefore  we now explore how to control the latter.

To motivates what follows let us observe that in the special case where $p=2$ we are basically dealing with a linear operator and one can try to recover standard energy estimates on a discrete level. For a PDE of type $w_{t}-w_{ss}=f$, an integral estimate for $w_{ts}$ is obtained by testing the equation differentiated in time with $w_{t}$. What follows employes essentially the same idea but on a discrete level.
  
\begin{lem}\label{controlVel}
Let $p=2$ and that $u_{0} \in W^{2,p}_{\rm per}(I)$.
Let the assumptions and notation of Theorem~\ref{thm:3.2} hold.  
Then 
\begin{align*}
\max_{t \in [0,T]}   \int_{I}| V_{n}(s,t)|^{2}  ds + \int_{0}^{T} \int_{I}|(V_{n}(s,t))_{s}|^{2} ds dt\leq  C \int_{0}^{T} \int_{I} |V_{n}(s,t)|^{2} ds dt+ \int_{I}| V_{0,n}(s)|^{2}  ds \leq C,
\end{align*}
where $V_{0,n}$ is defined by
\begin{align*}
V_{0,n}(s) := (u_{0}+\phi)_{ss}  + \tilde{\lambda}_{1} (u_{0}) \sin (u_{0} +\phi) -\tilde{\lambda}_{2} (u_{0}) \cos (u_{0} +\phi),
\end{align*}
and $C=C(p,c_{*},L, \|u_{0}\|_{W^{2,p}})$.
\end{lem}
\begin{proof}
In the following we keep the discussion as general as possible (in regard to the choice of $p$) so that it becomes visible where difficulties arise when $p \neq 2$.\\
From \eqref{firstvar} with $u=u_{i,n}\in W^{1,p}_{\rm per}(I)$ solution to \eqref{Min} we obtain
\begin{align*}
0&=\int_{I} \frac{(u_{i,n}(s)-u_{i-1,n}(s))}{\tau_{n}} \varphi(s) ds -\tilde{\lambda}_{1}(u_{i-1,n}) \int_{I} \sin (u_{i,n} + \phi) \varphi(s) ds \notag\\
&\qquad + \tilde{\lambda}_{2}(u_{i-1,n}) \int_{I} \cos (u_{i,n} + \phi) \varphi(s) ds 
  +\int_{I} \left( \left|(u_{i,n} +  \phi)_{s}\right|^{p-2} (u_{i,n} +  \phi)_{s} \right ) \vp_{s} ds\\
&=\int_{I} V_{i,n}(s) \varphi(s) ds -\tilde{\lambda}_{1}(u_{i-1,n}) \int_{I} \sin (u_{i,n} + \phi) \varphi(s) ds + \tilde{\lambda}_{2}(u_{i-1,n}) \int_{I} \cos (u_{i,n} + \phi) \varphi(s) ds \notag\\
& \qquad +\int_{I} \left( \left|(u_{i,n} +  \phi)_{s}\right|^{p-2} (u_{i,n} +  \phi)_{s} \right ) \vp_{s} ds
\end{align*}
for all $\varphi \in W^{1,p}_{\rm per}(I)$.
Replacing $i$ with $i-1$ we obtain an equation for $V_{i-1,n}$ (where $i \geq 2$). Subtraction of these two equations gives
then
\begin{align*}
0&=\int_{I} (V_{i,n}(s) - V_{i-1,n}(s))\varphi(s) ds \\
& \qquad +\int_{I} \left( \left|(u_{i,n} +  \phi)_{s}\right|^{p-2} (u_{i,n} +  \phi)_{s}  -   \left|(u_{i-1,n} +  \phi)_{s}\right|^{p-2} (u_{i-1,n} +  \phi)_{s} \right )  \vp_{s} ds\\
& \qquad 
-\tilde{\lambda}_{1}(u_{i-1,n}) \int_{I} \sin (u_{i,n} + \phi) \varphi(s) ds 
+ \tilde{\lambda}_{1}(u_{i-2,n}) \int_{I} \sin (u_{i-1,n} + \phi) \varphi(s) ds\\
& \qquad 
+ \tilde{\lambda}_{2}(u_{i-1,n}) \int_{I} \cos (u_{i,n} + \phi) \varphi(s) ds 
-\tilde{\lambda}_{2}(u_{i-2,n}) \int_{I} \cos (u_{i-1,n} + \phi) \varphi(s) ds 
\end{align*}
for all $\varphi \in W^{1,p}_{\rm per}(I)$. Choosing $\varphi = V_{i,n}(s)  \in W^{1,p}_{\rm per}(I)$  and using Lemma~\ref{lemDB} we get
\begin{align*}
0&\geq \int_{I} (V_{i,n}(s) - V_{i-1,n}(s)) V_{i,n}(s) ds + C\frac{1}{\tau_{n}} \int_{I}  |(u_{i,n} -u_{i-1,n})_{s}|^{p} ds\\
& \qquad 
-\tilde{\lambda}_{1}(u_{i-1,n}) \int_{I} \sin (u_{i,n} + \phi)  V_{i,n}(s) ds 
+ \tilde{\lambda}_{1}(u_{i-2,n}) \int_{I} \sin (u_{i-1,n} + \phi)  V_{i,n}(s) ds\\
& \qquad 
+ \tilde{\lambda}_{2}(u_{i-1,n}) \int_{I} \cos (u_{i,n} + \phi)  V_{i,n}(s) ds 
-\tilde{\lambda}_{2}(u_{i-2,n}) \int_{I} \cos (u_{i-1,n} + \phi)  V_{i,n}(s) ds .
\end{align*}
Using the simple equality $ a(a-b)=\frac{1}{2} a^{2} +\frac{1}{2}|a-b|^{2} - \frac{1}{2} b^{2}$ we can write
\begin{align*}
 &\frac{1}{2} \int_{I}| V_{i,n}(s)|^{2}  ds 
-\frac{1}{2} \int_{I}| V_{i-1,n}(s)|^{2}  ds
+ C\tau_{n}^{p-1} \int_{I}  |(V_{i,n})_{s}|^{p} ds\\
& \leq -\frac{1}{2} \int_{I} |V_{i,n}(s) - V_{i-1,n}(s)|^{2} ds     \\& \qquad 
+\tilde{\lambda}_{1}(u_{i-1,n}) \int_{I} \sin (u_{i,n} + \phi)  V_{i,n}(s) ds 
- \tilde{\lambda}_{1}(u_{i-2,n}) \int_{I} \sin (u_{i-1,n} + \phi)  V_{i,n}(s) ds\\
& \qquad 
- \tilde{\lambda}_{2}(u_{i-1,n}) \int_{I} \cos (u_{i,n} + \phi)  V_{i,n}(s) ds 
+\tilde{\lambda}_{2}(u_{i-2,n}) \int_{I} \cos (u_{i-1,n} + \phi)  V_{i,n}(s) ds \\
& \leq  |\tilde{\lambda}_{1}(u_{i-1,n}) - \tilde{\lambda}_{1}(u_{i-2,n})| \int_{I} |V_{i,n}(s)| ds
+ |\tilde{\lambda}_{1}(u_{i-2,n})| \int_{I} |u_{i,n}(s) -u_{i-1,n}(s)| |V_{i,n}(s)| ds\\
& \quad +  |\tilde{\lambda}_{2}(u_{i-1,n}) - \tilde{\lambda}_{2}(u_{i-2,n})| \int_{I} |V_{i,n}(s)| ds
+ |\tilde{\lambda}_{2}(u_{i-2,n})| \int_{I} |u_{i,n}(s) -u_{i-1,n}(s)| |V_{i,n}(s)| ds.
\end{align*}
By Theorem~\ref{thm:3.2} 
we have  uniform bounds for $\| (u_{j,n})_{s}\|_{L^{p}}$,  
$ |\tilde{\lambda}_{1} (u_{j,n})|$ and 
$|\tilde{\lambda}_{2} (u_{j,n})|$  with $j \in \{1, \ldots,n\}$.  Together with \eqref{LIPlambda},   we obtain for any $i \in \{ 2, \ldots, n \}$ and $\epsilon >0$ that
\begin{align}\label{VSUM}
 &\frac{1}{2} \int_{I}| V_{i,n}(s)|^{2}  ds 
-\frac{1}{2} \int_{I}| V_{i-1,n}(s)|^{2}  ds
+ C\tau_{n}^{p-1} \int_{I}  |(V_{i,n})_{s}|^{p} ds  \notag\\
& \leq C \tau_{n} \int_{I} |V_{i,n}(s)|^{2} ds + C  (\|(u_{i-1,n}- u_{i-2,n})_{s} \|_{L^{p}(I)} + \|u_{i-1,n}- u_{i-2,n} \|_{L^{2}(I)})\|V_{i,n} \|_{L^{2}(I)} \notag \\
& \leq C_{\epsilon} \tau_{n} \int_{I} |V_{i,n}(s)|^{2} ds + \epsilon \frac{1}{\tau_{n}}\|(u_{i-1,n}- u_{i-2,n})_{s} \|_{L^{p}(I)}^{2} +\frac{1}{2}\tau_{n} \int_{I} |V_{i-1,n}(s)|^{2} ds  \notag \\
&= C_{\epsilon} \tau_{n} \int_{I} |V_{i,n}(s)|^{2} ds + \epsilon \tau_{n} \|(V_{i-1,n})_{s}  \|_{L^{p}(I)}^{2} + \frac{1}{2} \tau_{n} \|V_{i-1,n}  \|_{L^{2}(I)}^{2}.  
\end{align}
Next  we need some information about $V_{1,n}$.  Using the regularity assumptions on the initial data we can define (recall here $p=2$)
\begin{align*}
V_{0,n}(s) := (|(u_{0}+\phi)_{s}|^{p-2}(u_{0}+\phi)_{s})_{s} + \tilde{\lambda}_{1} (u_{0}) \sin (u_{0} +\phi) -\tilde{\lambda}_{2} (u_{0}) \cos (u_{0} +\phi).
\end{align*}
Testing with $\varphi \in W^{1,p}_{\rm per}(I)$, integrating by parts, using the periodicity of $|(u_{0}+\phi)_{s}|^{p-2}(u_{0}+\phi)_{s}$ and recalling that by defintion $u_{0,n}=u_{0}$, we see that $V_{0,n}$ satisfies
 \begin{align*}
0&=\int_{I} V_{0,n}(s) \varphi(s) ds -\tilde{\lambda}_{1}(u_{0,n}) \int_{I} \sin (u_{0,n} + \phi) \varphi(s) ds + \tilde{\lambda}_{2}(u_{0,n}) \int_{I} \cos (u_{0,n} + \phi) \varphi(s) ds \notag\\
& \qquad +\int_{I}  (|(u_{0,n}+\phi)_{s}|^{p-2}(u_{0,n} +  \phi)_{s}  \vp_{s} ds
\end{align*}
for all $\varphi \in W^{1,p}_{\rm per}(I)$.
Subtracting the equation for $V_{1,n}$ from the above one we infer
\begin{align*}
0&=\int_{I} (V_{1,n}(s) - V_{0,n}(s))\varphi(s) ds \\
& \qquad +\int_{I} ((|(u_{1,n}+\phi)_{s}|^{p-2}(u_{1,n} +  \phi)_{s} -(|(u_{0,n}+\phi)_{s}|^{p-2}(u_{0,n} +  \phi)_{s})  \vp_{s} ds\\
& \qquad 
-\tilde{\lambda}_{1}(u_{0,n}) \int_{I} (\sin (u_{1,n} + \phi)  - \sin (u_{0,n} + \phi))\varphi(s) ds \\
& \qquad 
+ \tilde{\lambda}_{2}(u_{0,n}) \int_{I} (\cos (u_{1,n} + \phi)  - \cos (u_{0,n} + \phi))\varphi(s) ds 
\end{align*}
for all $\varphi \in W^{1,p}_{\rm per}(I)$. Choosing $\varphi = V_{1,n}(s)  \in W^{1,p}_{\rm per}(I)$ and arguing as above we get
\begin{align}\label{passo1}
 \frac{1}{2} \int_{I}| V_{1,n}(s)|^{2}  ds &
-\frac{1}{2} \int_{I}| V_{0,n}(s)|^{2}  ds
+ C\tau_{n}^{p-1} \int_{I}  |(V_{1,n})_{s}|^{p} ds \\
&\leq   C  \|u_{1,n}- u_{0,n} \|_{L^{2}(I)} \|V_{1,n} \|_{L^{2}(I)} 
 \leq C \tau_{n} \int_{I} |V_{1,n}|^{2} ds. \notag 
\end{align}
This gives in particular that
\begin{align*}
 \int_{I}| V_{1,n}(s)|^{2}  ds +C\tau_{n}^{p-1} \int_{I}  |(V_{1,n})_{s}|^{p} ds &\leq   \int_{I}| V_{0,n}(s)|^{2}  ds
 + C \tau_{n} \int_{I} |V_{1,n}|^{2} ds \\
 &\leq \int_{I}| V_{0,n}(s)|^{2} + C \int_{0}^{T} \int_{I} |V_{n}(s,t)|^{2} ds dt.
\end{align*}
Together with \eqref{VSUM} and $p=2$ the above inequality  \eqref{passo1} yields for any $j \in \{1, \ldots n \}$
\begin{align*}
\sum_{i=1}^{j} &\left ( \frac{1}{2} \int_{I}| V_{i,n}(s)|^{2}  ds 
-\frac{1}{2} \int_{I}| V_{i-1,n}(s)|^{2}  ds
+ C\tau_{n} \int_{I}  |(V_{i,n})_{s}|^{2} ds \right) \\
& \leq C_{\epsilon} \sum_{i=1}^{n} \tau_{n}\int_{I}| V_{i,n}(s)|^{2}  ds  + \epsilon
\tau_{n} \sum_{i=1}^{j-1} \int_{I}  |(V_{i,n})_{s}|^{2} ds 
\end{align*}
With $\epsilon $ small enough we finally infer
\begin{align*}
\max_{i=1, \ldots, n}   \int_{I}| V_{i,n}(s)|^{2}  ds +\int_{0}^{T} \int_{I}|(V_{n}(s,t))_{s}|^{2} ds dt &\leq  C \int_{0}^{T} \int_{I} |V_{n}(s,t)|^{2} ds dt+ \int_{I}| V_{0,n}(s)|^{2}  ds .
\end{align*}
By Theorem~\ref{thm:3.2} and the smoothness of the initial data we get the claim.
\end{proof}

\subsubsection{The case $p=2$: Control of the Lagrange multipliers}

 Application of Lemma~\ref{controlVel} gives information about the regularity and convergence of the maps~$\Lambda_{n}^{j}$.
\begin{lem}\label{lem:Lambda}  
Let $p=2$ and  $u_{0} \in W^{2,p}_{\rm per}(I)$. Let the assumptions and notation of Theorem~\ref{thm:3.2} hold. Then
we  have that $\Lambda_{n}^{j} \in W^{1,2}(0,T)$ with $\| \Lambda_{n}^{j}\|_{W^{1,2}(0,T)} \leq C$ for  $j=1,2$ and $\Lambda_{n}^{j}$ converges uniformly  (and weakly in $W^{1,2}(0,T)$) to a continous map $\Lambda_{j} \in H^{1}(0,T)$ for $j=1,2$.
Moreover   we have that $\Lambda_{j}=\lambda_{j}$ for $j=1,2$.
\end{lem}
\begin{proof}
With $p=2$ we can use  Lemma~\ref{lem:Lambdazero} and the bounds of Lemma~\ref{controlVel} to infer the statement on weak and uniform convergence.
To prove the last statement, consider  $\varphi \in L^{2}(0,T)$. Then we have for $j=1,2$:
\begin{align*}
\left |\int_{0}^{T} (\lambda_{j} -\Lambda_{j}) \varphi dt \right | \leq \left|\int_{0}^{T} (\lambda_{j}- \tilde{\lambda}^{j}_{n})\varphi   dt  \right| + \left
|\int_{0}^{T} ( \tilde{\lambda}^{j}_{n}-\Lambda_{n}^{j})\varphi   dt \right| + \left|\int_{0}^{T} ( \Lambda_{n}^{j} -\Lambda_{j})\varphi   dt \right| .
\end{align*}
The first and third integral on the right handside  go to zero for $n \to \infty$ on account of weak convergence (recall Lemma~\ref{lem:3.3}). For the second one observe that
\begin{align*}
\left|\int_{0}^{T} ( \tilde{\lambda}^{j}_{n}-\Lambda_{n}^{j})\varphi   dt \right| &= \left|\sum_{i=1}^{n} \int_{(i-1)\tau_{n}}^{i\tau_{n}} (t -(i-1)\tau_{n}) \frac{\tilde{\lm}_{j}(u_{i,n})-\tilde{\lm}_{j}(u_{i-1,n})}{\tau_{n}} \varphi  dt \right| \\& \leq  \tau_{n} \int_{0}^{T} |(\Lambda_{n}^{j})_{t}(t)||\varphi(t)| dt
\end{align*}
which goes to zero using the bounds of Lemma~\ref{controlVel}.
\end{proof}

\subsubsection{The case $p>2$: Control of the velocities for small initial data}

For the case $p>2$ we will be able to infer a control on the velocities   provided the initial data is small (in a sense that will be made precise below). We first derive some useful estimates. 
\begin{lem}\label{diffuinfty}
  Let $p \geq 2$ and 
assume that there exists a map $V_{0,n} \in L^{2}(I)$
 such that  
 \begin{align*}
0&=\int_{I} V_{0,n}(s) \varphi(s) ds -\tilde{\lambda}_{1}(u_{0,n}) \int_{I} \sin (u_{0,n} + \phi) \varphi(s) ds + \tilde{\lambda}_{2}(u_{0,n}) \int_{I} \cos (u_{0,n} + \phi) \varphi(s) ds \notag\\
& \qquad +\int_{I}  |(u_{0,n}+\phi)_{s}|^{p-2}(u_{0,n} +  \phi)_{s}  \vp_{s} ds
\end{align*}
for all $\varphi \in W^{1,p}_{\rm per}(I)$. Let the assumptions of Theorem~\ref{thm:3.2} hold. Set $V_{-1,n}:=0$. Then we have that for any $i=1, \ldots,n$ the bounds
\begin{align*}
\Big | |(u_{i,n}+\phi)_{s} |^{ p-1} -|(u_{i-1,n}+\phi)_{s} |^{ p-1}\Big | &\leq C \Big (
\| V_{i,n}-V_{i-1,n} \|_{L^{2}(I)} \\
& \qquad +   \tau_{n} \| (V_{i,n})_{s}  \|_{L^{2}(I)} +\tau_{n} \| V_{i,n}  \|_{L^{2}(I)}
\Big ) \qquad \quad  \text{ if $i=1$,}\\
\Big | |(u_{i,n}+\phi)_{s} |^{ p-1} -|(u_{i-1,n}+\phi)_{s} |^{ p-1}\Big | &\leq C \Big (
\| V_{i,n}-V_{i-1,n} \|_{L^{2}(I)} +   \tau_{n} \| (V_{i,n})_{s}  \|_{L^{2}(I)}
 + \tau_{n} \| V_{i,n}  \|_{L^{2}(I)} \\
 & \qquad + \tau_{n} \| (V_{i-1,n})_{s}  \|_{L^{2}(I)}+ \tau_{n} \| V_{i-1,n}  \|_{L^{2}(I)}
\Big ) \qquad \text{ if $i >1$,}\\
\Big | |(u_{i,n}+\phi)_{s} |^{ p-1} -|(u_{0,n}+\phi)_{s} |^{ p-1}\Big | &\leq C \Big (
\| V_{i,n}\|_{L^{2}(I)}+ \|V_{0,n} \|_{L^{2}(I)} \\
& \qquad +  \sum_{j=1}^{i}  \tau_{n} \| (V_{j,n})_{s}  \|_{L^{2}(I)}
+  \sum_{j=1}^{i}  \tau_{n} \| V_{j,n}  \|_{L^{2}(I)}
\Big )
\end{align*}
hold on $I$, provided  there exists  a constant $\hat{c}$ so that  $  \| V_{j,n} \|_{L^{2}(I)}\leq \hat{c}$ for all $j=0, \ldots,i$.
Here $C=C(\hat{c}, c^{*}, L, p)$. 
\end{lem}
\begin{proof}
Set $w_{i}:=|(u_{i,n}+\phi)_{s}|^{p-2} (u_{i,n}+\phi)_{s}$, $i=0, \ldots,n$.
We have that 
\begin{align*}
\Big | |(u_{i,n}+\phi)_{s} |^{ p-1} -|(u_{i-1,n}+\phi)_{s} |^{ p-1}\Big | & =
\Big | |w_{i}| - | w_{i-1}|   \Big | \leq \| w_{i}-w_{i-1}\|_{L^{\infty}}.
\end{align*}
The same arguments (and notation) employed to derive Lemma~\ref{lemma3bis} (recall \eqref{regeq}) yield that
\begin{align*}
(w_{i})_{s} &= V_{i,n} -\tilde{\lambda}_{1}(u_{i-1,n}) \sin (u_{i,n}+\phi) + \tilde{\lambda}_{2}(u_{i-1,n}) \cos (u_{i,n}+\phi) 
\end{align*}
for all $i=0, \ldots, n$ where we have set $u_{-1,n}:=u_{0,n}$. 
Note also  that by  the second inequality  in Lemma~\ref{Lindquist}, the fourth inequality in Lemma~\ref{lemma3bis} and the assumption 
$\| V_{i,n} \|_{L^{2}(I)}$, $ \| V_{i-1,n} \|_{L^{2}(I)} \leq \hat{c}$ we have
\begin{align}\label{muffin1}
|w_{i}-w_{i-1}| \leq C |(u_{i,n}-u_{i-1,n})_{s}|
\end{align}
with $C=C(\hat{c},p, c^{*},L)$. Therefore, using embedding theory, we can write
\begin{align*}
\| w_{i}-w_{i-1}\|_{L^{\infty}} &\leq \frac{1}{L}\int_{I} |w_{i}-w_{i-1}| ds + \int_{I} |(w_{i}-w_{i-1})_{s}| ds \\
& \leq C \int_{I} |(u_{i,n}-u_{i-1,n})_{s}| + |V_{i,n}-V_{i-1,n}| + |\tilde{\lambda}_{1}(u_{i-1,n}) - \tilde{\lambda}_{1}(u_{i-2,n})|ds\\
& \qquad + C\int_{I}|\tilde{\lambda}_{2}(u_{i-1,n}) - \tilde{\lambda}_{2}(u_{i-2,n})| + |u_{i,n}-u_{i-1,n}| ds\\
& \leq  C  \left (\| u_{i,n}-u_{i-1,n}\|_{W^{1,2}(I)} +  \| V_{i,n}-V_{i-1,n} \|_{L^{2}(I)} +  C \tau_{n} \| V_{i-1,n} \|_{W^{1,2}(I)}
\right )
\end{align*}
where we have used \eqref{lemconti}, the mean value theorem, and \eqref{est-lambda-forte} in the last inequality. The first two claims now follow.
The last claim is obtained in a similar way. More precisely we observe that
\begin{align*}
\Big | |(u_{i,n}+\phi)_{s} |^{ p-1} -|(u_{0,n}+\phi)_{s} |^{ p-1}\Big | & =
\Big | |w_{i}| - | w_{0}|   \Big | \leq \| w_{i}-w_{0}\|_{L^{\infty}} \\&
\leq
\frac{1}{L}\int_{I} |w_{i}-w_{0}| ds + \int_{I} |(w_{i}-w_{0})_{s}| ds,
\end{align*}
as well as
\begin{align*}
|w_{i}-w_{0}| \leq \sum_{j=1}^{i} |w_{j}-w_{j-1}| \leq C \sum_{j=1}^{i}  |(u_{j,n}-u_{j-1,n})_{s}|
\end{align*}
by \eqref{muffin1}. For the derivative we write instead
\begin{align*}
(w_{i}-w_{0})_{s}& =
V_{i,n}  -V_{0,n}  +
 \sum_{j=1}^{i} \Big ( -\tilde{\lambda}_{1}(u_{j-1,n}) \sin (u_{j,n}+\phi)  + \tilde{\lambda}_{1}(u_{j-2,n}) \sin (u_{j-1,n}+\phi) \Big) \\
 & \quad + \sum_{j=1}^{i} \Big (\tilde{\lambda}_{2}(u_{j-1,n}) \cos (u_{j,n}+\phi) - \tilde{\lambda}_{2}(u_{j-2,n}) \cos (u_{j-1,n}+\phi) \Big)
\end{align*}
and then argue as above.
\end{proof}

With more information about the bounds for $u_{i,n}$ we can now extend the results of Lemma~\ref{controlVel} to a wider class of $p$. To do that we will assume some smallness assumption on the initial data.
In that respect note that if we take $u_{0}=0$, then $\theta_{0}(s)=\phi(s)=\frac{2\pi \eta}{L}s$, which corresponds to an initial planar curve $\gamma_{0}=\gamma_{0}(s)$ (parametrized by arc length) with tangent $\gamma_{0}'(s)= (\cos \theta_{0}(s), \sin \theta_{0}(s))$, $s \in [0,L]$. In other words $\theta_{0}$ corresponds to a   circle of radius $L/2\pi \eta$ (multiply covered in case $ \eta>1$). Note also that if $u_{0}=0$, then $\tilde{\lambda}_{r}(u_{0})=0$ for $r=1,2$ and $V_{0,n}=0$ with $V_{0,n}$ defined as in \eqref{defVzeron}.
Next we show that we can control the velocities on a possibly smaller time interval $[0,T^{*}] \subset [0,T]$.

\begin{lem}\label{lem:controlVelpBig} Let $p \geq 2$. Let the assumptions and notation of Theorem~\ref{thm:3.2} hold.
 Let $u_{0} \in W^{1,p}_{\rm per}(I) $ be such that the weak derivative $(|(u_{0}+\phi)_{s}|^{p-2}(u_{0}+\phi)_{s})_{s}$ exists and belongs to $L^{2}(I)$ and $[|(u_{0}+\phi)_{s}|^{p-2}(u_{0}+\phi)_{s}]_{0}^{L}=0$.
Define
\begin{align}\label{defVzeron}
V_{0,n}(s) := (|(u_{0}+\phi)_{s}|^{p-2}(u_{0}+\phi)_{s})_{s} + \tilde{\lambda}_{1} (u_{0}) \sin (u_{0} +\phi) -\tilde{\lambda}_{2} (u_{0}) \cos (u_{0} +\phi),
\end{align}
and set $c_{0}:=(\frac{1}{2})^{\frac{p-2}{p-1}}|\phi_{s}|^{p-2}=(\frac{1}{2})^{\frac{p-2}{p-1}}|\frac{2\pi \eta}{L}|^{p-2}$. 
Assume that the initial data $u_{0}$ is so small so that  
we have 
\begin{align}\label{SC} \tag{SC}
\left \{ \begin{array}{l}
c_{*}=2 \|( u_{0} + \phi)_{s} \|_{L^{p}(I)}^{p} \leq 4 \| \phi_{s} \|_{L^{p}(I)}^{p}\\
\| V_{0,n}\|_{L^{2}(I)}^{2}  \leq \delta \leq \min \{1/2, \tilde{C} \}\\ 
    |(u_{0} + \phi)_{s}|^{p-1} \geq \frac{1}{2} |\phi_{s}|^{p-1} \quad\text{almost everywhere on } I,
    \end{array}  \right .
\end{align}
with $\tilde{C}$  a fixed constant that depends on $c_{0}$, $L$, $p$.
Let  $0<T< \delta $ and let $n \geq n_{0}=n_{0}(c_{0},L,p)$ be sufficiently large.
Then 
$$ |(u_{0} + \phi)_{s}|^{p-2} \geq  c_{0} \quad \text{ almost everywhere on } I, $$ 
and 
for any $i=1, \ldots, n$ we have
$$\| V_{i,n}\|_{L^{2}(I)} \leq 1, \qquad  |(u_{i,n}+\phi)_{s} |^{ p-1} \geq\frac{1}{2} (\frac{1}{2}|\phi_{s}|^{p-1}) \quad\text{almost everywhere on } I,$$
and
\begin{align*}
  \int_{I}| V_{i,n}(s)|^{2}  ds 
+&  \frac{ c_{0}}{2}  \sum_{j=1}^{i}\tau_{n} \int_{I}  |(V_{j,n})_{s}|^{2} ds  
+ \sum_{j=1}^{i}\int_{I} |V_{j,n}(s) - V_{j-1,n}(s)|^{2} ds  
 \\
& \leq C \left(\int_{I}| V_{0,n}(s)|^{2}  ds + \sum_{j=1}^{i-1}  \tau_{i-1} \int_{I} |V_{j,n}(s)|^{2} ds \right ) \leq 2C\delta,
\end{align*}
where $C$ is a constant that depends only on $c_{0}$, $L$, $p$.
\end{lem}
\begin{proof}
We use an induction argument.
From \eqref{firstvar} with $u=u_{i,n}\in W^{1,p}_{\rm per}(I)$ solution to \eqref{Min} we obtain
\begin{align*}
0&=\int_{I} V_{i,n}(s) \varphi(s) ds -\tilde{\lambda}_{1}(u_{i-1,n}) \int_{I} \sin (u_{i,n} + \phi) \varphi(s) ds + \tilde{\lambda}_{2}(u_{i-1,n}) \int_{I} \cos (u_{i,n} + \phi) \varphi(s) ds \notag\\
& \qquad +\int_{I} \left( \left|(u_{i,n} +  \phi)_{s}\right|^{p-2} (u_{i,n} +  \phi)_{s} \right ) \vp_{s} ds
\end{align*}
for all $\varphi \in W^{1,p}_{\rm per}(I)$.
Replacing $i$ with $i-1$ we obtain an equation for $V_{i-1,n}$. Subtraction of these two equations gives
then (here $i \geq2$)
\begin{align*}
0&=\int_{I} (V_{i,n}(s) - V_{i-1,n}(s))\varphi(s) ds \\
& \qquad +\int_{I} \left( \left|(u_{i,n} +  \phi)_{s}\right|^{p-2} (u_{i,n} +  \phi)_{s}  -   \left|(u_{i-1,n} +  \phi)_{s}\right|^{p-2} (u_{i-1,n} +  \phi)_{s} \right )  \vp_{s} ds\\
& \qquad 
-\tilde{\lambda}_{1}(u_{i-1,n}) \int_{I} \sin (u_{i,n} + \phi) \varphi(s) ds 
+ \tilde{\lambda}_{1}(u_{i-2,n}) \int_{I} \sin (u_{i-1,n} + \phi) \varphi(s) ds\\
& \qquad 
+ \tilde{\lambda}_{2}(u_{i-1,n}) \int_{I} \cos (u_{i,n} + \phi) \varphi(s) ds 
-\tilde{\lambda}_{2}(u_{i-2,n}) \int_{I} \cos (u_{i-1,n} + \phi) \varphi(s) ds 
\end{align*}
for all $\varphi \in W^{1,p}_{\rm per}(I)$. Choosing $\varphi = V_{i,n}(s)  \in W^{1,p}_{\rm per}(I)$  and using the first inequality in  Lemma~\ref{Lindquist} we get
\begin{align*}
0&\geq \int_{I} (V_{i,n}(s) - V_{i-1,n}(s)) V_{i,n}(s) ds + \frac{1}{2}\frac{1}{\tau_{n}} \int_{I} |(u_{i-1,s}+\phi)_{s}|^{p-2} |(u_{i,n} -u_{i-1,n})_{s}|^{2} ds\\
& \qquad 
-\tilde{\lambda}_{1}(u_{i-1,n}) \int_{I} \sin (u_{i,n} + \phi)  V_{i,n}(s) ds 
+ \tilde{\lambda}_{1}(u_{i-2,n}) \int_{I} \sin (u_{i-1,n} + \phi)  V_{i,n}(s) ds\\
& \qquad 
+ \tilde{\lambda}_{2}(u_{i-1,n}) \int_{I} \cos (u_{i,n} + \phi)  V_{i,n}(s) ds 
-\tilde{\lambda}_{2}(u_{i-2,n}) \int_{I} \cos (u_{i-1,n} + \phi)  V_{i,n}(s) ds .
\end{align*}
Using the simple equality $ a(a-b)=\frac{1}{2} a^{2} +\frac{1}{2}|a-b|^{2} - \frac{1}{2} b^{2}$ we can write
\begin{align*}
 &\frac{1}{2} \int_{I}| V_{i,n}(s)|^{2}  ds 
-\frac{1}{2} \int_{I}| V_{i-1,n}(s)|^{2}  ds
+ \frac{1}{2}\tau_{n} \int_{I} |(u_{i-1,s}+\phi)_{s}|^{p-2}  |(V_{i,n})_{s}|^{2} ds\\
& +\frac{1}{2} \int_{I} |V_{i,n}(s) - V_{i-1,n}(s)|^{2} ds     \\& \leq
\tilde{\lambda}_{1}(u_{i-1,n}) \int_{I} \sin (u_{i,n} + \phi)  V_{i,n}(s) ds 
- \tilde{\lambda}_{1}(u_{i-2,n}) \int_{I} \sin (u_{i-1,n} + \phi)  V_{i,n}(s) ds\\
& \qquad 
- \tilde{\lambda}_{2}(u_{i-1,n}) \int_{I} \cos (u_{i,n} + \phi)  V_{i,n}(s) ds 
+\tilde{\lambda}_{2}(u_{i-2,n}) \int_{I} \cos (u_{i-1,n} + \phi)  V_{i,n}(s) ds \\
& \leq  |\tilde{\lambda}_{1}(u_{i-1,n}) - \tilde{\lambda}_{1}(u_{i-2,n})| \int_{I} |V_{i,n}(s)| ds
+ |\tilde{\lambda}_{1}(u_{i-2,n})| \int_{I} |u_{i,n}(s) -u_{i-1,n}(s)| |V_{i,n}(s)| ds\\
& \quad +  |\tilde{\lambda}_{2}(u_{i-1,n}) - \tilde{\lambda}_{2}(u_{i-2,n})| \int_{I} |V_{i,n}(s)| ds
+ |\tilde{\lambda}_{2}(u_{i-2,n})| \int_{I} |u_{i,n}(s) -u_{i-1,n}(s)| |V_{i,n}(s)| ds.
\end{align*}
By Theorem~\ref{thm:3.2}  
we have  uniform bounds for $\| (u_{j,n})_{s}\|_{L^{p}}$,  
$ |\tilde{\lambda}_{1} (u_{j,n})|$ and 
$|\tilde{\lambda}_{2} (u_{j,n})|$  with $j \in \{1, \ldots,n\}$. Note that thanks to the first assumption in \eqref{SC} these bounds do  not depend on $u_{0}$ any longer. Together with Lemma~\ref{lemma3bis} (fourth statement), and \eqref{est-lambda-forte},    we obtain for any $i \in \{ 2, \ldots, n \}$ and $\epsilon >0$ that
\begin{align*}
 &\frac{1}{2} \int_{I}| V_{i,n}(s)|^{2}  ds 
-\frac{1}{2} \int_{I}| V_{i-1,n}(s)|^{2}  ds
+\frac{1}{2}\tau_{n} \int_{I} |(u_{i-1,s}+\phi)_{s}|^{p-2}  |(V_{i,n})_{s}|^{2} ds \notag\\
& +\frac{1}{2} \int_{I} |V_{i,n}(s) - V_{i-1,n}(s)|^{2} ds  \notag \\
& \leq C \tau_{n} \int_{I} |V_{i,n}(s)|^{2} ds + C \tau_{n}(1+ \|V_{i-1,n}\|_{L^{2}(I)} + \|V_{i-2,n}\|_{L^{2}(I)}) \|(V_{i-1,n})_{s}\|_{L^{2}(I)} \|V_{i,n}\|_{L^{2}(I)} \notag\\
& \qquad +  C \tau_{n}\|V_{i,n}\|_{L^{2}(I)} \|V_{i-1,n}\|_{L^{2}(I)} \notag \\
&   \leq C(1+\frac{1}{\epsilon}) \tau_{n} \int_{I} |V_{i,n}(s)|^{2} ds +  \epsilon \tau_{n}(1+ \|V_{i-1,n}\|_{L^{2}(I)} + \|V_{i-2,n}\|_{L^{2}(I)})^{2} \|(V_{i-1,n})_{s}\|_{L^{2}(I)}^{2}\\
& \qquad +  \frac{1}{2}\tau_{n}\|V_{i-1,n}\|_{L^{2}(I)}^{2} .
\end{align*}
We proceed with an induction argument.
If we know that
\begin{align*}
|(u_{j,n}+\phi)_{s} |^{ p-1} \geq\frac{1}{2} (\frac{1}{2}|\phi_{s}|^{p-1}) \quad\text{almost everywhere on } I,\\
|(u_{j,n}+\phi)_{s} |^{ p-2} \geq (\frac{1}{2})^{\frac{p-2}{p-1}} c_{0}\geq \frac{c_{0}}{2} \quad\text{almost everywhere on } I,\\
\|V_{j,n} \|_{L^{2}} \leq 1, 
\end{align*}
for all $j=0, \ldots, i-1$ 
and set $\epsilon=\frac{1}{3^{2}\cdot 8} c_0 $ we obtain
\begin{align}\label{VSUM2bis}
 &\frac{1}{2} \int_{I}| V_{i,n}(s)|^{2}  ds 
-\frac{1}{2} \int_{I}| V_{i-1,n}(s)|^{2}  ds 
+\frac{1}{4 } c_{0} \tau_{n} \int_{I}  |(V_{i,n})_{s}|^{2} ds \\
& \quad +\frac{1}{2} \int_{I} |V_{i,n}(s) - V_{i-1,n}(s)|^{2} ds  \notag \\
&   \leq C \tau_{n} \int_{I} |V_{i,n}(s)|^{2} ds +   \tau_{n} \frac{c_{0}}{8} \|(V_{i-1,n})_{s}\|_{L^{2}(I)}^{2} +  \frac{1}{2}\tau_{n}\|V_{i-1,n}\|_{L^{2}(I)}^{2} .\notag
\end{align}
Summming \eqref{VSUM2bis} up over $j=1, \ldots,i$  (for the initial step $j=1$ see \eqref{passo1bis} below) and using that $T < \delta$ we get for $1 \leq i \leq n$:
\begin{align*}
 &\frac{1}{2} \int_{I}| V_{i,n}(s)|^{2}  ds 
- \frac{1}{2}\int_{I}| V_{0,n}(s)|^{2}  ds
+  \frac{ c_{0}}{4}  \tau_{n} \int_{I}  |(V_{i,n})_{s}|^{2} ds  + \sum_{j=1}^{i-1}  \frac{ c_{0}}{8}  \tau_{n} \int_{I}  |(V_{j,n})_{s}|^{2} ds\notag\\
& 
+ \frac{1}{2}\sum_{j=1}^{i}\int_{I} |V_{j,n}(s) - V_{j-1,n}(s)|^{2} ds  
\leq C\tau_{n} \int_{I} |V_{i,n}(s)|^{2} ds + \sum_{j=1}^{i-1} C \tau_{n} \int_{I} |V_{j,n}(s)|^{2} ds  \\
& \leq C\tau_{n} \int_{I} |V_{i,n}(s)|^{2} ds+ C n\tau_{n} \leq C\tau_{n} \int_{I} |V_{i,n}(s)|^{2} ds + C\delta.
\end{align*}
With $\tau_{n} \leq \frac{1}{4C}$ we can absorb the integral term on the right hand-side and get the claimed estimates.
Moreover $\|V_{i,n} \|_{L^{2}} \leq 1$ also follows  for $\tilde{C}$ appropiately chosen,
so that by the third claim in Lemma~\ref{diffuinfty} we infer
\begin{align*}
\Big | |(u_{i,n}+\phi)_{s} |^{ p-1} -|(u_{0,n}+\phi)_{s} |^{ p-1}\Big | &\leq C \Big (
\| V_{i,n}\|_{L^{2}(I)}+ \|V_{0,n} \|_{L^{2}(I)} \\
& \qquad +  \sum_{j=1}^{i}  \tau_{n} \| (V_{j,n})_{s}  \|_{L^{2}(I)}
+ \sum_{j=1}^{i}  \tau_{n} \| V_{j,n}  \|_{L^{2}(I)}
\Big )
\end{align*}
and therefore
\begin{align*}
& |(u_{i,n}+\phi)_{s} |^{ p-1} 
\geq |(u_{0,n}+\phi)_{s} |^{ p-1} - \Big   | |(u_{i,n}+\phi)_{s} |^{ p-1} -|(u_{0,n}+\phi)_{s} |^{ p-1}\Big |\\
& \quad \geq \frac{1}{2}|\phi_{s}|^{p-1} - C ( \| V_{i,n}\|_{L^{2}(I)} + \| V_{0,n} \|_{L^{2}(I)} +\sum_{j=1}^{i}   \tau_{n} \| (V_{j,n})_{s} \|_{L^{2}(I)} + \sum_{j=1}^{i}  \tau_{n} \| V_{j,n}  \|_{L^{2}(I)}
)\\
&\quad \geq  \frac{1}{2}|\phi_{s}|^{p-1} - C  \sqrt{\delta}  \geq\frac{1}{4} |\phi_{s}|^{p-1}
\end{align*} 
where we have used that
\begin{align*}
&\sum_{j=1}^{i}  \tau_{n} \| V_{j,n}  \|_{L^{2}(I)}
+\sum_{j=1}^{i} \tau_{n} \| (V_{j,n})_{s} \|_{L^{2}(I)}  \leq (i\tau_{n})^{1/2} (\sum_{j=1}^{i}\tau_{n} \| V_{j,n} \|_{L^{2}(I)}^{2} )^{1/2}\\
& \qquad + (i\tau_{n})^{1/2} (\sum_{j=1}^{i}\tau_{n} \| (V_{j,n})_{s} \|_{L^{2}(I)}^{2} )^{1/2} \leq \sqrt{T} C\sqrt{\delta},\\
&\| V_{i,n}\|_{L^{2}(I)} + \| V_{0,n} \|_{L^{2}(I)} \leq C \sqrt{\delta},
\end{align*}
and provided $\delta $ is small enough (thus we may have to take $\tilde{C}$ even smaller). This gives
$$ |(u_{i,n}+\phi)_{s} |^{ p-1} \geq  \frac{1}{2}(\frac{1}{2}|\phi_{s}|^{p-1}) \text{ and } |(u_{i,n}+\phi)_{s} |^{ p-2} \geq (\frac{1}{2})^{\frac{p-2}{p-1}} c_{0}\geq \frac{c_{0}}{2}.$$
Next  we need some information about $V_{1,n}$, in order to show the first induction step.  Recall that by assumption 
\begin{align*}
V_{0,n}(s) = (|(u_{0}+\phi)_{s}|^{p-2}(u_{0}+\phi)_{s})_{s} + \tilde{\lambda}_{1} (u_{0}) \sin (u_{0} +\phi) -\tilde{\lambda}_{2} (u_{0}) \cos (u_{0} +\phi).
\end{align*}
Testing with $\varphi \in W^{1,p}_{\rm per}(I)$, integrating by parts, using the periodicity of $|(u_{0}+\phi)_{s}|^{p-2}(u_{0}+\phi)_{s}$ and recalling that by defintion $u_{0,n}=u_{0}$, we see that $V_{0,n}$ satisfies
 \begin{align*}
0&=\int_{I} V_{0,n}(s) \varphi(s) ds -\tilde{\lambda}_{1}(u_{0,n}) \int_{I} \sin (u_{0,n} + \phi) \varphi(s) ds + \tilde{\lambda}_{2}(u_{0,n}) \int_{I} \cos (u_{0,n} + \phi) \varphi(s) ds \notag\\
& \qquad +\int_{I}  |(u_{0,n}+\phi)_{s}|^{p-2}(u_{0,n} +  \phi)_{s}  \vp_{s} ds
\end{align*}
for all $\varphi \in W^{1,p}_{\rm per}(I)$.
Subtracting the equation for $V_{1,n}$ from the above one we infer
\begin{align*}
0&=\int_{I} (V_{1,n}(s) - V_{0,n}(s))\varphi(s) ds \\& \quad +\int_{I} ((|(u_{1,n}+\phi)_{s}|^{p-2}(u_{1,n} +  \phi)_{s} -(|(u_{0,n}+\phi)_{s}|^{p-2}(u_{0,n} +  \phi)_{s})  \vp_{s} ds\\
& \qquad 
-\tilde{\lambda}_{1}(u_{0,n}) \int_{I} (\sin (u_{1,n} + \phi)  - \sin (u_{0,n} + \phi))\varphi(s) ds \\
& \qquad 
+ \tilde{\lambda}_{2}(u_{0,n}) \int_{I} (\cos (u_{1,n} + \phi)  - \cos (u_{0,n} + \phi))\varphi(s) ds 
\end{align*}
for all $\varphi \in W^{1,p}_{\rm per}(I)$. Choosing $\varphi = V_{1,n}(s)  \in W^{1,p}_{\rm per}(I)$, applying Lemma~\ref{Lindquist} as above and using the smallness of the initial data (that guarantees  that $|(u_{0,n}+\phi)_{s}|^{p-2} \geq c_{0}>0$) we get
\begin{align}\label{passo1bis}
 \frac{1}{2} \int_{I}| V_{1,n}(s)|^{2}  ds &
-\frac{1}{2} \int_{I}| V_{0,n}(s)|^{2}  ds
+ \frac{c_{0}}{8}\tau_{n} \int_{I}  |(V_{1,n})_{s}|^{2} ds
+ \frac{1}{2} \int_{I}| V_{1,n}(s) -V_{0,n} (s)|^{2}  ds \notag \\
&\leq   C  \|u_{1,n}- u_{0,n} \|_{L^{2}(I)} \|V_{1,n} \|_{L^{2}(I)} 
 \leq C \tau_{n} \int_{I} |V_{1,n}|^{2} ds.
\end{align}
With $\frac{1}{2}-C\tau_{n} \geq \frac{1}{4}$ we infer that
\begin{align*}
\int_{I}| V_{1,n}(s)|^{2}  ds &+\frac{c_{0}}{2}\tau_{n} \int_{I}  |(V_{1,n})_{s}|^{2} ds  +  \int_{I}| V_{1,n}(s) -V_{0,n} (s)|^{2}  ds \\
 &\leq  2 \int_{I}| V_{0,n}(s)|^{2}  ds \leq C\delta
\end{align*}
 Using now that $\|V_{1,n}\|_{L^{2}}^{2},\|V_{0,n}\|_{L^{2}}^{2} \leq 1 $, 
and Lemma~\ref{diffuinfty} we obtain that
\begin{align*}
\Big | |(u_{1,n}+\phi)_{s} |^{ p-1} -|(u_{0,n}+\phi)_{s} |^{ p-1}\Big | &\leq C \Big (
\| V_{1,n}-V_{0,n} \|_{L^{2}(I)} \\
& \qquad +  C \tau_{n} \| (V_{1,n})_{s} \|_{L^{2}(I)} + C \tau_{n} \| V_{1,n} \|_{L^{2}(I)}
\Big )  \leq C\sqrt{\delta},
\end{align*}
and therefore
\begin{align*}
|(u_{1,n}+\phi)_{s} |^{ p-1} &\geq |(u_{0,n}+\phi)_{s} |^{ p-1} - \Big | |(u_{1,n}+\phi)_{s} |^{ p-1} -|(u_{0,n}+\phi)_{s} |^{ p-1}\Big |\\
& \geq \frac{1}{2}|\phi_{s}|^{p-1} - C \sqrt{\delta}  \geq \frac{1}{4} |\phi_{s}|^{p-1}
\end{align*}
provided $\sqrt{\delta} \leq \frac{1}{4C}|\phi_{s}|^{p-1}$. Hence, collecting all estimates, we have obtained 
\begin{align*}
&|(u_{1,n}+\phi)_{s} |^{ p-1} \geq\frac{1}{2} (\frac{1}{2}|\phi_{s}|^{p-1})\quad\text{almost everywhere on } I,\\
&|(u_{1,n}+\phi)_{s} |^{ p-2} \geq |\phi_{s}|^{p-2}(\frac{1}{4})^{\frac{p-2}{p-1}} 
= c_{0}  (\frac{1}{2})^{\frac{p-2}{p-1}} \geq \frac{c_{0}}{2}\quad\text{almost everywhere on } I,\\
&\| V_{0,n} \|_{L^{2}(I)}, \| V_{1,n} \|_{L^{2}(I)} \leq1 .
\end{align*}
\end{proof}
 
\subsubsection{The case $p>2$: Control of the Lagrange multipliers for small initial data}

With the help of \eqref{est-lambda-forte} and Lemma~\ref{lem:controlVelpBig} we can now infer new information about the Lagrange multipliers (recall Definiton \ref{defBiglambdas}) when $p>2$.
\begin{lem} \label{lem:lambda2plus}
Assume $p\geq2$.
Let the assumptions and notation of Theorem \ref{thm:3.2} and Lemma~\ref{lem:controlVelpBig} hold. 
Then  for $\Lambda^{j}_{n}$, $j=1,2$, we have
\begin{align}
&|(\Lambda_{n}^{j})_{t}(t)| \leq C \| V_{n}( \cdot, t)\|_{W^{1,2}(I)} \qquad \text{for a.e.} t \in [0,T], \\
& |\Lambda_{n}^{j}(t)| \leq C \qquad \text{for all } t \in [0,T],
\end{align}
where $C=C(p,c_{*},L, u_{0})$.\\
Furthermore we  have that $\Lambda_{n}^{j} \in W^{1,2}(0,T)$ with $\| \Lambda_{n}^{j}\|_{W^{1,2}(0,T)} \leq C$ for all $n\geq n_{0}$ and $\Lambda_{n}^{j}$ converges unifomly  (and weakly in $W^{1,2}(0,T)$) to a continous map $\Lambda_{j} \in H^{1}(0,T)$ for $j=1,2$.
Moreover  we have that $\Lambda_{j}=\lambda_{j}$ for $j=1,2$ on $(0, T)$.
\end{lem}
\begin{proof}
Following the definition of $\Lambda_{n}^{j}$ and using \eqref{est-lambda-forte} and Lemma~\ref{lem:controlVelpBig} it follows that 
\begin{align*}
|(\Lambda_{n}^{j})_{t}(t)|= \left|   \frac{\tilde{\lm}_{j}(u_{i,n})-\tilde{\lm}_{j}(u_{i-1,n})}{\tau_{n}}  \right| \leq C\| V_{n}( \cdot, t)\|_{W^{1,2}(I)} \qquad \text{ for } t \in ((i-1)\tau_{n}, i \tau_{n}).
\end{align*}
The other claims follow exactly as in Lemma~\ref{lem:Lambda} now using Lemma~\ref{lem:controlVelpBig} instead of Lemma~\ref{controlVel}.
\end{proof}

 
\subsection{Convergence procedure}
\label{sec:3.2}
\subsubsection{Convergence to a weak solution of \eqref{eq:P} ($1<p<\infty$)}

\begin{lem} \label{lem:4.1}
Let $1< p < \infty$. 
Let $u_{0} \in W^{1,p}_{\rm per}(I)$, $c_{*}$, and $T=T(p,u_{0},L)$ be as in Theorem~\ref{thm:3.2}.
Let $u_{n}$ be the piecewise linear interpolation of $\{u_{i,n}\}$ given in Definition~\ref{definition:2.4}. 
Then there exists a map 
\begin{align*}
u \in L^{\infty}(0,T; W^{1,p}_{\rm per}(I)) \cap H^{1}(0,T; L^{2}(I)) 
\end{align*}
such that  
\begin{align} \label{eq:4.1}
\int^{T}_{0} \int_{I} | \pd_t u(s,t) |^{2} \, dx dt < \frac{4}{p} c_{*},
\end{align}

\begin{align}\label{eq:4.1bis}
\sup_{(0,T)} \| u \|_{W^{1,p}(I)} \leq C=C(c_{*}, \| \phi \|_{W^{1,p}},p, \| u_{0}\|_{L^{2}(I)}),
\end{align}

and (for a subsequence which we still denote by $u_{n}$)
\begin{align} \label{eq:4.2}
\begin{cases}
&u_{n} \rightharpoonup u \quad \text{weakly\,$\ast$ in} \quad L^{\infty}(0,T; W^{1,p}(I)), \\
&u_{n} \rightharpoonup u \quad \text{weakly in} \quad  H^{1}(0,T; L^{2}(I)), 
\end{cases}
\quad \text{as} \quad n \to \infty. 
\end{align}
Moreover, 
for $\alpha = \min \{ \frac{1}{4}, \frac{p-1}{2p}\}$  we have that 
\begin{align} \label{eq:4.3}
u_{n} \to u \quad \text{in} \quad  C^{0,\va}([0,T] \times I). 
\end{align}
In particular, $u(\cdot,t) \to u_{0}(\cdot)$ in $C^{0}$ as $t \downarrow 0$. 
\end{lem}

\begin{proof}
First of all notice that by Theorem \ref{thm:3.2}   we have that $u_{n}(\cdot,t)  \in W^{1,p}_{\rm per}(I)$ and
\begin{align} \label{eq:4.6old}
\sup_{t \in (\,0, T\,)} (\| (u_{n}(\cdot,t))_{s} \|_{L^{p}(I)}+ \| u_{n}(\cdot,t) \|_{L^{2}(I)})
 \le  C(c_{*}, \| \phi \|_{W^{1,p}(I)},p, \| u_{0}\|_{L^{2}(I)}).
\end{align} 
Lemma~\ref{Hilfsatz1} yields 
\begin{align}\label{eq:4.6}
\sup_{t \in (\,0, T\,)} \| u_{n}(\cdot,t) \|_{W^{1,p}(I)}
 \le  C(c_{*}, \| \phi \|_{W^{1,p}},p, \| u_{0}\|_{L^{2}(I)}).
 \end{align} 
Thus there exists a map $u \in L^{\infty}(0,T; W^{1,p}(I))$ such that 
\begin{align*}
u_{n} \rightharpoonup u \quad \text{weakly$\ast$ in} \quad L^{\infty}(0,T; W^{1,p}(I)) \quad \text{as} \quad n \to \infty 
\end{align*}
up to a subsequence.

Next note that since $u_{n}(s, \cdot)$ is absolutely continuous in $[\,0, T\,]$, 
for any $t_{1}$, $t_{2} \in [\,0, T\,]$ with $t_{1}< t_{2}$, we infer from H\"older's inequality and Fubini's theorem that  
\begin{align*}
\| u_{n}(\cdot, t_{2}) - u_{n}(\cdot, t_{1}) \|_{L^{2}(I)} 
 &= \left( \int_{I} \av{\int^{t_{2}}_{t_{1}} \dfrac{\pd u_{n}}{\pd t}(s,\tau) \, d\tau}^{2}\, ds \right)^{\frac{1}{2}} \\
 &\le \left( \int^{t_{2}}_{t_{1}} \int_{I} \av{\dfrac{\pd u_{n}}{\pd t}(s,\tau)}^{2}\, ds d \tau \right)^{\frac{1}{2}} (t_{2}-t_{1})^{\frac{1}{2}}. 
\end{align*} 
According to Theorem \ref{thm:3.2}, we find 
\begin{align} \label{eq:4.4}
\int^{t_{2}}_{t_{1}} \int_{I} \av{\dfrac{\pd u_{n}}{\pd t}(s,\tau)}^{2}\, ds d \tau
 = \int^{t_{2}}_{t_{1}} \int_{I} \av{V_{n}(s,\tau)}^{2} \, ds d \tau 
 \le \frac{4}{p}c_{*}.  
\end{align}
Hence we obtain 
\begin{align} \label{eq:4.5}
\| u_{n}(\cdot, t_{2}) - u_{n}(\cdot, t_{1}) \|_{L^{2}(I)} \le \sqrt{\frac{4}{p}c_{*}} (t_{2}-t_{1})^{\frac{1}{2}}. 
\end{align}

We turn to the proof of \eqref{eq:4.3}. 
First of all observe that by \eqref{eq:4.6} and embedding theory we have that
\begin{align}\label{linftyb}
\sup_{ t \in [0,T]} \| u_{n} (\cdot, t) \|_{L^{\infty}(I)} \leq C, \qquad \text{ with } C=C(L,c_{*},p, \| \phi\|_{W^{1,p}(I)},\| u_{0}\|_{L^{2}(I)}).
\end{align}
Moreover for any $t \in [0,T]$ we have
\begin{align}\label{pip}
|u_{n}(s_{2},t) - u_{n}(s_{1},t)| \leq |\int_{s_{1}}^{s_{2}}   (u_{n})_{s}(s, t)ds| \leq C |s_{2}-s_{1}|^{1-\frac{1}{p}}
\end{align}
with $C=C(c_{*},p, \| \phi\|_{W^{1,p}(I)})$.

Fix $0 \le t_1 \le t_2 \le T$ arbitrarily and set $$\Gm(\cdot):= u_n(\cdot, t_2) - u_n(\cdot, t_1) \in W^{1,p}_{\rm per}(I).$$ 
By an interpolation inequality (see for instance \cite[Thm. 5.9]{Adams})
we find  
\begin{align*}
\| \Gm \|_{L^\infty} \le C \| \Gm \|_{L^q}^{1/2} \|  \Gm \|_{W^{1,p}}^{1/2}, \qquad \text{ where } \frac{1}{p} + \frac{1}{q}=1. 
\end{align*}
Using \eqref{eq:4.6} yields
\begin{align*}
\| \Gm \|_{L^\infty} \le C \| \Gm \|_{L^q}^{1/2}, \qquad \text{ where } \frac{1}{p} + \frac{1}{q}=1 .
\end{align*}
Now  if $ 1 \leq q \leq 2$ then $\| \Gm \|_{L^{q}} \leq C \| \Gm \|_{L^{2}}$ for some $C=C(L)$ and by \eqref{eq:4.5} we infer
$$\| \Gm \|_{L^\infty} \le C | t_2 - t_1 |^{1/4} $$
with $C=C(L,p,c_{*}, \| \phi\|_{W^{1,p}(I)}, \|u_{0}\|_{L^{2}})$.
On the other hand if $ 2<q< \infty$ (i.e. $1<p<2$) another interpolation inequality (see for instance \cite[Appendix B2]{Evans}) gives
$$ \| \Gm \|_{L^{q}} \leq  \| \Gm \|_{L^{2}}^{\theta} \| \Gm \|_{L^{\infty}}^{1-\theta} $$
for $\frac{1}{q}=\frac{\theta}{2}$ (note that $\Gm \in L^{\infty}$,  
recall \eqref{linftyb}).
Then we obtain
 $$\| \Gm \|_{L^\infty} \le C \| \Gm \|_{L^q}^{1/2} \leq C\| \Gm \|_{L^{2}}^{\theta/2} \| \Gm \|_{L^{\infty}}^{\frac{1-\theta}{2}} \leq C\| \Gm \|_{L^{2}}^{\theta/2}= C\| \Gm \|_{L^{2}}^{\frac{p-1}{p}} .$$ 
Using \eqref{eq:4.5} we find 
\begin{align} \label{eq:4.9}
\| \Gm \|_{L^\infty} \le C | t_2 - t_1 |^{\frac{p-1}{2p} }.
\end{align}

Hence, it follows  that 
\begin{align*} 
\| u_{n}(t_{2}) - u_{n}(t_{1}) \|_{C^{0}} \le C | t_{2} - t_{1} |^\alpha \qquad \text{ with } \alpha =1/4 \text{ if }  2 \leq p < \infty, \text{ and } \alpha = \frac{p-1}{2p} \text{ otherwise}. 
\end{align*}
Since $\frac{p-1}{2p} \geq \frac{1}{4}$ if and only if $p \geq 2$ then we can state that
\begin{align} \label{eq:4.12}
\| u_{n}(t_{2}) - u_{n}(t_{1}) \|_{C^{0}} \le C | t_{2} - t_{1} |^\alpha \qquad \text{ with } \alpha = \min \{ \frac{1}{4}, \frac{p-1}{2p}\}. 
\end{align}
Next, since $\alpha \leq \frac{p-1}{2p} \leq \frac{p-1}{p}$ we infer from the above inequality and \eqref{pip} that
\begin{align}
|u_{n}( s_{2}, t_{2}) - u_{n}(s_{1}, t_{1})| \leq C (|t_{2}-t_{1}|^{\alpha} + |s_{2}-s_{1}|^{\alpha})
\end{align}
for any $(t_{i}, s_{i}) \in [0,T] \times [0,L]$, $i=1,2$.
Application of the Arzel\`a-Ascoli theorem yields \eqref{eq:4.3}. In particular $u ( \cdot, t) \in W^{1,p}_{\rm per}(I)$ for all times.
Moreover, setting $t_{1}=0$ in \eqref{eq:4.12}, we observe that 
\begin{align*}
\| u(t) - u_{0} \|_{C^{0}} \to 0 \quad \text{as} \quad t \downarrow 0. 
\end{align*}

Finally, from \eqref{eq:4.4}  we also infer that  there exists $V \in L^{2}(0,T; L^{2}(I)) $ such that
\begin{align}\label{VH}
V_{n} = \pd_{t} u_{n} \rightharpoonup V  \quad \text{in} \quad L^{2}(0,T; L^{2}(I)). 
\end{align}
Moreover we see that for any $v \in C_{0}^{\infty}((0,T) \times I)$ 
\begin{align*}
\int_{0}^{T} \int_{I} u_{n} v_{t}\,  ds dt = -\int_{0}^{T}\int_{I}  V_{n}  v  \, dsdt \to  -\int_{0}^{T} \int_{I} V v \, dsdt
\end{align*}
and using the fact  that $u_{n} \to u$ uniformly continuous
\begin{align*}
\int_{0}^{T}\int_{I} u_{n} v_{t} \, ds dt \to \int_{0}^{T} \int_{I} u v_{t} \, ds dt
\end{align*}
from which we infer that $u$ admits weak derivative $u_{t}=V$.
Hence we have $u \in H^{1}(0,T; L^{2}(I))$ and \eqref{eq:4.2}. 
Moreover, it follows from \eqref{eq:4.4} that \eqref{eq:4.1} holds. 
\end{proof}


\begin{lem} \label{lem:4.2}
Let $\tilde{u}_{n}$, $\tilde{U}_{n}$ be the piecewise constant interpolations of $\{ u_{i,n} \}$ as given in Definition~\ref{definition:2.5} and let the assumptions of Lemma~\ref{lem:4.1} hold. 
Then we have 
\begin{align} \label{eq:4.13}
\tilde{u}_{n} \to u  \quad \text{ and } \quad  \tilde{U}_{n} \to u \text{ in } C^{0}([0,T] \times I). 
\end{align} 
where $u$ denotes the map obtained in Lemma \ref{lem:4.1}.   
Moreover, it holds that 
\begin{align} \label{eq:4.14}
 (\tilde{u}_{n})_{s} \rightharpoonup  u_{s} \quad  \text{ and } \quad  (\tilde{U}_{n})_{s} \rightharpoonup  u_{s} \quad  \text{weakly in} \quad L^{p}(0,T; L^{p}(I)) \quad \text{as} \quad n \to \infty. 
\end{align}
\end{lem}

\begin{proof}
We show the proof only for $\tilde{u}_{n}$, since analogous arguments holds for $\tilde{U}_{n}$.
Recalling \eqref{eq:4.6old} and \eqref{eq:4.6}, we  see that $\tilde{u}_{n} \in L^{\infty}(0,T; W^{1,p}_{\rm per}(I))$. 
In particular 
equations \eqref{linftyb}, \eqref{pip} now hold with $u_{n}$ replaced by $\tilde{u}_{n}$. 
 
Fix $t \in (\,0, T\,]$ arbitrarily. 
Then there exists a family of intervals $\{ (\,(i_{n}-1) \tau_{n}, i_{n} \tau_{n} \,] \}_{n \in \N}$ such that 
$t \in (\,(i_{n}-1) \tau_{n}, i_{n} \tau_{n} \,]$. 
We deduce from \eqref{eq:4.12} that 
\begin{align*}
\| \tilde{u}_{n}(t) - u_{n}(t) \|_{C^{0}(I)} 
& = \| u_{i_{n},n} - u_{n}(t) \|_{C^{0}(I)} 
 = \| u_{n}(i_{n} \tau_{n}) - u_{n}(t) \|_{C^{0}(I)} \\
& \le C |i_{n} \tau_{n} - t|^{\alpha} 
 \le C \tau_{n}^{\va} \to 0 \quad \text{as} \quad n \to \infty  
\end{align*} 
Since $u_{n} \to u$ in $C^{0}([0,T] \times I)$ by  Lemma~\ref{lem:4.1} and $t$ was arbitrarily chosen, we infer that $\tilde{u}_{n} \to u$ in $C^{0}([0,T] \times I)$. 

We turn to \eqref{eq:4.14}. 
Again recalling \eqref{eq:4.6}, we also see that $(\tilde{u}_{n})_{s} \in L^{p}(0,T; L^{p}(I))$ and $\| (\tilde{u}_{n})_{s}\|_{L^{p}(0,T; L^{p}(I))} \leq C$ for all $n \in \mathbb{N}$. Now $L^{p}(0,T; L^{p}(I))$ is a reflexive Banach space therefore there exists $ v \in L^{p}(0,T; L^{p}(I))$ such that $ (\tilde{u}_{n})_{s} \rightharpoonup  v$.
This implies that 
$$ \int^{T}_{0} \int_{I} (\tilde{u}_{n})_{s} \cdot \vp  \, dsdt \to  \int^{T}_{0} \int_{I}  v \cdot \vp \, dsdt $$
for any $\vp \in L^{q}(0,T; L^{q}(I))$, with $1/p + 1/q=1$.
On the other hand if $\vp \in L^{\infty}(0,T; C_{0}^{\infty}(I))$ we infer that
$$  \int^{T}_{0} \int_{I}  (\tilde{u}_{n})_{s} \vp \,dsdt = - \int^{T}_{0} \int_{I}  \tilde{u}_{n} \cdot \vp_{s} \, dsdt \to 
 - \int^{T}_{0} \int_{I}  u \cdot \vp_{s} \, dsdt $$
where we have used that $\tilde{u}_{n} \to u$. Hence we derive that $v=u_{s}$ and the second claim follows.
\end{proof}

\begin{lem} \label{lem:4.4}
Let $1<p<\infty$.
Let $\tilde{u}_{n}$ be the piecewise constant interpolation of $\{ u_{i,n} \}$ as given in Definition~\ref{definition:2.5} and let the assumptions of Lemma~\ref{lem:4.1} hold. 
Then,  
it holds that 
\begin{align*}
\int^{T}_{0} \int_{I} | (\tilde{u}_{n} + \phi)_{s}|^{p-2}  (\tilde{u}_{n} + \phi)_{s}  \cdot  \vp_{s} \, ds dt 
 \to \int^{T}_{0} \int_{I} | (u+\phi)_{s} |^{p-2}  (u+\phi)_{s}  \cdot  \vp_{s} \, ds dt 
 \quad \text{as} \quad n \to \infty 
\end{align*}
for any $\vp \in L^{\infty}(0,T; W^{1,p}_{\rm per}(I))$, where $u$ denotes the map obtained in Lemma \ref{lem:4.1}. 

 Moreover we have that the map $| (u+\phi)_{s} |^{p-2}  (u+\phi)_{s} \in L^{\infty}(0,T; L^{\frac{p}{p-1}}(I))$ admits for almost every time a weak derivative (in the space variable)
\begin{align*}
(| (u+\phi)_{s} |^{p-2}  (u+\phi)_{s})_{s} \in L^{2}(0,T; L^{2}(I)), 
\end{align*}
and  $[| (u+\phi)_{s} |^{p-2}  (u+\phi)_{s}]_{0}^{L}=0$ for a. e. $t \in (0,T)$.
\end{lem}
\begin{proof}
According to Lemma \ref{lemma3bis} and Theorem~\ref{thm:3.2} we see that 
\begin{align*}
\int^{T}_{0} \|  | (\tilde{u}_{n} + \phi)_{s}|^{p-2}  (\tilde{u}_{n} + \phi)_{s}  \|_{W^{1,2}(I)}^{2} \, dt 
 \le C \int^{T}_{0} ( 1 + \| V_{n} \|^{2}_{L^{2}} ) \, dt 
 \le C. 
\end{align*}
Thus we find $w \in L^{2}(0,T; W^{1,2}(I))$ such that 
\begin{align} \label{eq:4.18}
 | (\tilde{u}_{n} + \phi)_{s}|^{p-2}  (\tilde{u}_{n} + \phi)_{s} \rightharpoonup w \quad \text{in} \quad 
L^{2}(0,T; W^{1,2}(I)) \quad \text{as} \quad n \to \infty 
\end{align}
 and
$\| w \|_{L^{2}(0,T; W^{1,2}(I))} \leq C=C(c_{*}, p,T)$.

This implies in particular that
\begin{align}\label{peppa}
\int_{0}^{T} \int_{I} | (\tilde{u}_{n} + \phi)_{s}|^{p-2}  (\tilde{u}_{n} + \phi)_{s} \cdot \vp \, ds dt \to \int_{0}^{T} \int_{I} w \cdot \vp \, ds dt  \quad \forall \vp \in L^{2}(0,T; L^{2}(I)).
\end{align}

On the other hand by Lemma \ref{lemma3bis} and Theorem~\ref{thm:3.2} we also have for $1/p+1/q=1$ that 
\begin{align*}
\int^{T}_{0} \|  | (\tilde{u}_{n} + \phi)_{s}|^{p-2}  (\tilde{u}_{n} + \phi)_{s}  \|_{L^{q}(I)}^{2} \, dt 
 \le C \int^{T}_{0} ( 1 + \| V_{n} \|^{2}_{L^{2}} ) \, dt 
 \le C.
\end{align*}
The space $L^{2}(0,T; L^{q}(I))$ is reflexive with dual space given by $(L^{2}(0,T; L^{q}(I)))^{*}= L^{2}(0,T; L^{p}(I))$. Hence there exists $\tilde{\xi} \in L^{2}(0,T; L^{q}(I))$ such that
\begin{align}\label{peppa2}
\int_{0}^{T} \int_{I} | (\tilde{u}_{n} + \phi)_{s}|^{p-2}  (\tilde{u}_{n} + \phi)_{s} \cdot \vp \, ds dt \to \int_{0}^{T} \int_{I} \tilde{\xi} \cdot \vp \, ds dt  \quad \forall \vp \in L^{2}(0,T; L^{p}(I)).
\end{align}
Together with \eqref{peppa} we infer that $w=\tilde{\xi}$ almost everywhere and $w \in L^{2}(0,T; W^{1,2}(I)) \cap L^{2}(0,T; L^{q}(I))= L^{2}(0,T; W^{1,2}(I))$.

Note also that the periodicity property
\begin{align}\label{periodicityproperty2}
0=[| (\tilde{u}_{n} + \phi)_{s}|^{p-2}  (\tilde{u}_{n} + \phi)_{s}]_{0}^{L} =| (\tilde{u}_{n} + \phi)_{s}|^{p-2}  (\tilde{u}_{n} + \phi)_{s}(L) -| (\tilde{u}_{n} + \phi)_{s}|^{p-2}  (\tilde{u}_{n} + \phi)_{s}(0)
\end{align}
that holds for all time (recall \eqref{periodicityproperty}) carries over to $w$ in the sense that for almost every time 
$w(L,t)-w(0,t)=0$. This follows from \eqref{peppa} testing with $\varphi(t,s)=\zeta(t)\psi_{s}(s)$ for $\psi \in W^{1,p}_{\rm per}(I)$, $\zeta \in C^{\infty}_{0}(0,T)$, integrating by parts and using \eqref{eq:4.18} and \eqref{peppa2} .

Next, set 
\begin{align*}
F(\psi) := \dfrac{1}{p} \int^{T}_{0} \int_{I} | ( \psi +\phi)_{s} |^{p} \, ds dt = \frac{1}{p}\| ( \psi +\phi)_{s} \|_{L^{p}(0,T; L^{p}(I))}^{p}.  
\end{align*}
Using the convexity of the $C^{1}$-map $y \to \frac{1}{p} |y|^{p}$, we observe that 
\begin{align} \label{eq:4.19}
F(\psi) - F(\tilde{u}_{n})  
 \ge \int^{T}_{0} \int_{I} |  (\tilde{u}_{n}+\phi)_{s}|^{p-2}  (\tilde{u}_{n}+\phi)_{s} \cdot (\psi - \tilde{u}_{n})_{s} \, dx dt 
 \quad \text{for any} \quad \psi \in L^{\infty}(0,T; W^{1,p}_{\rm per}(I)). 
\end{align}
Recalling \eqref{eq:4.14} and the fact that $F(\cdot)$ is weakly lower semicontinuous and letting $n \to \infty$ in~\eqref{eq:4.19}, we have   
\begin{align} \label{eq:4.20}
F(\psi) - F(u) \ge \int^{T}_{0} \int_{I} w \cdot (\psi - u)_{s} \, dx dt. 
\end{align}
Indeed, using \eqref{eq:4.18}, \eqref{eq:4.13}, \eqref{periodicityproperty2} we see that    
\begin{align*}
& \int^{T}_{0} \int_{I}  | (\tilde{u}_{n} + \phi)_{s}|^{p-2}  (\tilde{u}_{n} + \phi)_{s}\cdot (\psi-\tilde{u}_{n})_{s} \, ds dt \\
& \quad = -\int^{T}_{0} \int_{I} ( | (\tilde{u}_{n} + \phi)_{s}|^{p-2}  (\tilde{u}_{n} + \phi)_{s})_{s} \cdot (\psi-\tilde{u}_{n}) \, ds dt\\
& \quad \to -\int^{T}_{0} \int_{I}  w_{s} \cdot (\psi- u) \, ds dt
 = \int^{T}_{0} \int_{I} w \cdot (\psi- u)_{s} \, ds dt,   
\end{align*}
where in the last inequality we have used the fact that $w \in L^{2}(0,T; W^{1,2}(I)) \cap L^{2}(0,T; L^{q}(I))$ and that $[w]_{0}^{L}=0$ for almost every time.

Setting $\psi= u + \ve \vp$ in \eqref{eq:4.20}  for some $\vp \in L^{\infty}(0,T; W^{1,p}_{\rm per}(I))$ and $\ve >0$, we obtain 
\begin{align} \label{eq:4.21}
\dfrac{F(u + \ve \vp) - F(u)}{\ve} \ge \int^{T}_{0} \int_{I} w \cdot  \vp_{s} \, ds dt.
\end{align}
On the other hand, putting $\psi= u - \ve \vp$ in \eqref{eq:4.20}, we obtain 
\begin{align} \label{eq:4.22}
\dfrac{F(u) - F(u - \ve \vp)}{\ve} \le \int^{T}_{0} \int_{I} w \cdot  \vp_{s} \, ds dt.
\end{align}
Plugging \eqref{eq:4.22} into \eqref{eq:4.21} and letting $\ve \downarrow 0$ (the existence of the Gateaux derivative is standard, see for instance \cite[Thm. 3.11 Step2]{Dacorogna}), we find  
\begin{align}\label{peppa4}
\int^{T}_{0} \int_{I} | (u + \phi)_{s}|^{p-2}  (u + \phi)_{s} \cdot  \vp_{s} \, ds dt 
 = \int^{T}_{0} \int_{I} w \cdot  \vp_{s} \, ds dt \quad \forall \, \vp \in L^{\infty}(0,T; W^{1,p}_{\rm per}(I)). 
\end{align}
Together with \eqref{peppa2} we obtain the first claim.
Next, observe that \eqref{eq:4.18} implies  also that
\begin{align}
\int_{0}^{T} \int_{I} (| (\tilde{u}_{n} + \phi)_{s}|^{p-2}  (\tilde{u}_{n} + \phi)_{s})_{s} \cdot \vp \, ds dt \to \int_{0}^{T} \int_{I} w_{s} \cdot \vp \, ds dt  \quad \forall \vp \in L^{2}(0,T; L^{2}(I)).
\end{align}
In particular choosing $\vp \in L^{2}(0,T; C^{\infty}_{\rm per}(I))$ and using \eqref{periodicityproperty2} we get 
\begin{align}
\int_{0}^{T} \int_{I} (| (\tilde{u}_{n} + \phi)_{s}|^{p-2} & (\tilde{u}_{n} + \phi)_{s})_{s} \cdot \vp \, ds dt  \notag \\
&= - 
\int_{0}^{T} \int_{I} | (\tilde{u}_{n} + \phi)_{s}|^{p-2}  (\tilde{u}_{n} + \phi)_{s} \cdot \vp_{s} \, ds dt
 \to -\int_{0}^{T} \int_{I} w \cdot \vp_{s} \, ds dt.
\end{align}
Together with \eqref{peppa4} this gives 
\begin{align}
\int^{T}_{0} \int_{I} | (u + \phi)_{s}|^{p-2}  (u + \phi)_{s} \cdot  \vp_{s} \, ds dt =- \int_{0}^{T} \int_{I} w_{s} \cdot \vp \, ds dt \qquad \forall \vp \in L^{2}(0,T; C^{\infty}_{\rm per}(I)),
\end{align} 
so that we infer that for almost every time  $| (u + \phi)_{s}|^{p-2}  (u + \phi)_{s}$ admits weak derivative  $(| (u + \phi)_{s}|^{p-2}  (u + \phi)_{s})_{s}=w_{s}$ and $[| (u + \phi)_{s}|^{p-2}  (u + \phi)_{s}]_{0}^{L}=0$.
\end{proof}


\begin{lem} \label{lem:4.4bis}
Let $1<p<\infty$.
Let $\tilde{U}_{n}$ be the piecewise constant interpolation of $\{ u_{i,n} \}$ as given in Definition~\ref{definition:2.5} and let the assumptions of Lemma~\ref{lem:4.1} hold. 
Then for any $\delta \in (0,T)$ we have that 
\begin{align*}
\int^{T}_{\delta} \int_{I} | (\tilde{U}_{n} + \phi)_{s}|^{p-2}  (\tilde{U}_{n} + \phi)_{s}  \cdot  \vp_{s} \, ds dt 
 \to \int^{T}_{\delta} \int_{I} | (u+\phi)_{s} |^{p-2}  (u+\phi)_{s}  \cdot  \vp_{s} \, ds dt 
 \quad \text{as} \quad n \to \infty 
\end{align*}
for any $\vp \in L^{\infty}(0,T; W^{1,p}_{\rm per}(I))$, where $u$ denotes the map obtained in Lemma \ref{lem:4.1}. 
\end{lem}
\begin{proof}
The proof follows with similar arguments used in Lemma~\ref{lem:4.4}. The main difference is that now we want to ``avoid'' the time interval $[0, \tau_{n}]$ because here $\tilde{U}_{n}(s,t)=u_{0}(s)$ and for the initial data $u_{0}$ we do not have the regularity properties derived  in Lemma~\ref{lemma3bis} (nor does \eqref{periodicityproperty} holds). Note that for given $\delta>0$ we have that $\tau_{n} < \delta$ if $n$ is sufficiently large.
\end{proof}


\begin{thm}\label{thm:existence}
Let $1< p < \infty$.  
Let $u_{0} \in W^{1,p}_{\rm per}$, $c_{*}$, and $T=T(p,u_{0},L)$ be as in Theorem~\ref{thm:3.2}. Then 
there exists a map $u \in L^{\infty}(0,T; W^{1,p}_{\rm per}(I)) \cap H^{1}(0,T; L^{2}(I))$  that satisfies properties (i), (ii)  and (iv) of Definition~\ref{def:WS} and the estimates  \eqref{eq:4.1}, \eqref{eq:4.1bis}.
 Moreover we have that the map $| (u+\phi)_{s} |^{p-2}  (u+\phi)_{s} \in L^{\infty}(0,T; L^{\frac{p}{p-1}}(I))$ admits for almost every time a weak derivative (in space) 
\begin{align*}
(| (u+\phi)_{s} |^{p-2}  (u+\phi)_{s})_{s}=u_{t} -\lambda_{1}(t) \sin (u+ \phi)  + \lambda_{2}(t) \cos(u+ \phi) \in L^{2}(0,T; L^{2}(I)), 
\end{align*}
and $[| (u + \phi)_{s}|^{p-2}  (u + \phi)_{s}]_{0}^{L}=0$ for a. e. $t \in (0,T)$.
\end{thm}
\begin{proof}
Equation \eqref{firstvar} and Theorem~\ref{thm:3.1} yield that for any $\varphi \in L^{\infty} (0,T; W^{1,p}_{\rm per}(I))$  and (for almost every) $t \in  ((i-1)\tau_{n},  i \tau_{n}]$, $i=1, \ldots, n$ we have
\begin{align*}
0&=\int_{I}V_{n}(s,t) \varphi(s,t) ds -\tilde{\lambda}^{1}_{n}(t) \int_{I} \sin (\tilde{u}_{n} +\phi) \varphi ds + \tilde{\lambda}^{2}_{n}(t)\int_{I} \cos (\tilde{u}_{n} +\phi) \varphi ds \\
& \qquad + \int_{I} |(\tilde{u}_{n} +\phi)_{s}|^{p-2} (\tilde{u}_{n} +\phi)_{s} \varphi_{s} ds,
\end{align*}
so that integration in time yields
\begin{align*}
0=& \int_{0}^{T}\int_{I}V_{n}(s,t) \varphi(s,t) ds dt  -\int_{0}^{T}\tilde{\lambda}^{1}_{n}(t) \int_{I} \sin (\tilde{u}_{n} +\phi) \varphi ds dt + \int_{0}^{T}\tilde{\lambda}^{2}_{n}(t)\int_{I} \cos (\tilde{u}_{n} +\phi) \varphi ds dt\\
& \qquad + \int_{0}^{T}\int_{I} |(\tilde{u}_{n} +\phi)_{s}|^{p-2} (\tilde{u}_{n} +\phi)_{s} \varphi_{s} ds dt \qquad \text{  for any } \varphi \in L^{\infty} (0,T; W^{1,p}_{\rm per}(I)).
\end{align*}
We now let $n \to \infty$. The last integral is dealt with in Lemma~\ref{lem:4.4}, the first one in Lemma~\ref{lem:4.1}.
By Lemma~\ref{lem:3.3} we have that there exist $\lambda_{1}, \lambda_{2} \in L^{2}(0,T)$ such that 
$$\tilde{\lambda}^{j}_{n} \rightharpoonup \lambda_{j} \text{ weakly in } L^{2}(0,T).$$
Since $v_{n}(t):=\int_{I} \sin (\tilde{u}_{n} +\phi) \varphi(s,t) ds \to \int_{I} \sin (u +\phi) \varphi(s,t) ds =:v(t) $ by  Lemma~\ref{lem:4.2}, and $|v_{n}| \leq C(\vp)$, then also $v_{n} \to v$ in $L^{2}(0,T)$ and 
we infer that
$$\int_{0}^{T}\tilde{\lambda}^{1}_{n}(t) \int_{I} \sin (\tilde{u}_{n} +\phi) \varphi ds dt   \to \int_{0}^{T} \lambda_{1}(t)\int_{I} \sin (u +\phi) \varphi(s,t) ds dt $$
for $n \to \infty$. The integral with the $\tilde{\lambda}^{2}_{n}$  is treated in a similar way.

Finally note that
since the solution $u$  fullfills (ii) in Definition~\ref{def:WS}, we have  that by taking $\varphi (t,s)= \tilde{\varphi}(t)\psi(s)$ with $\tilde{\varphi}\in C^{\infty}_{0}(0,T)$ and $\psi \in W^{1,p}_{\rm per}(I)$
\begin{align}\label{bella}
\int_{I} \Big( u_{t} -\lambda_{1}(t) \sin (u+ \phi)  + \lambda_{2}(t) \cos(u+ \phi) \Big) \psi +  
 \left|u_{s} +  2\pi \frac{\eta}{L}\right|^{p-2} (u_{s} +  2\pi \frac{\eta}{L})  \psi_{s}
                                    \, ds = 0
\end{align}
for almost every time $t \in [0,T]$ and for any map $ \psi \in W^{1,p}_{\rm per}(I)$. In other words $\int_{I}\xi \psi + w\psi_{s} ds=0$ for all $\psi \in W^{1,p}_{\rm per}(I)$, where $\xi: =u_{t} -\lambda_{1}(t) \sin (u+ \phi)  + \lambda_{2}(t) \cos(u+ \phi) \in L^{2}(I)$
and $w:=\left|u_{s} +  2\pi \frac{\eta}{L}\right|^{p-2} (u_{s} +  2\pi \frac{\eta}{L}) \in L^{\frac{p}{p-1}}$.
This gives that $w_{s}=\xi$. 
\end{proof}
It remains to show that property $(iii)$ of Definition~\ref{def:WS} is also fulfilled by the map $u$ of Theorem~\ref{thm:existence}. In other words  we would like to show $\frac{d}{dt} \int_{I}T(u+\phi) ds=0$ (where $T$ is as in \eqref{defiT} with $\theta=u+\phi$). 
First we show  that  the discrete Lagrange multipliers  fulfill the expected property \eqref{vorrei} in the limit equation for almost every time.
\begin{lem}\label{lem:3.16}
Let the assumptions of Theorem~\ref{thm:existence} hold. Then we have that
\begin{align*}
 \int_{0}^{T}\varphi(t)\int_{I} \Big \langle \left (\begin{array}{c} \lambda_{1}(t)\\ \lambda_{2}(t) \end{array} \right), &
 \left (\begin{array}{c} -\sin (u +\phi)\\ \cos (u +\phi) \end{array} \right) \Big \rangle   \left (\begin{array}{c} -\sin (u +\phi)\\ \cos (u +\phi) \end{array} \right)ds dt\\
 &= 
 \int_{0}^{T}\varphi(t)\int_{I} |(u+\phi)_{s}|^{p}  \left (\begin{array}{c} \cos (u +\phi)\\ \sin (u +\phi) \end{array} \right)ds dt\\
 &  = \int_{0}^{T} \varphi(t)
 \int_{I} (|(u+\phi)_{s}|^{p-2} (u+\phi)_{s})_{s} \left (\begin{array}{c} -\sin (u +\phi)\\ \cos (u +\phi) \end{array} \right)ds dt   
 \end{align*}
 for any $\varphi \in C^{\infty}_{0}(0,T)$.
In particular $\lambda_{j}(t)= \tilde{\lambda}_{j}(u(\cdot, t))$ for $j=1,2$ and almost every $t \in (0,T)$. 
\end{lem}
\begin{proof}  The integral expression follows from multiplying by $\varphi$ and
integrating in time expression \eqref{eqRuleL} and passing to the limit (recall Lemma~\ref{lem:3.3}, Lemma~\ref{lem:4.2} and Lemma~\ref{lem:4.4bis}). Integration by parts is then used for the second equality (recall the second statement in Lemma~\ref{lem:4.4}).
We show here some details for the most delicate term, that is we want to show the convergence
\begin{align*}
 \int_{0}^{T}\varphi(t)\int_{I} |(\tilde{U}_{n}+\phi)_{s}|^{p}  \left (\begin{array}{c} \cos (\tilde{U}_{n} +\phi)\\ \sin (\tilde{U}_{n} +\phi) \end{array} \right)ds dt\to  \int_{0}^{T}\varphi(t)\int_{I} |(u+\phi)_{s}|^{p}  \left (\begin{array}{c} \cos (u +\phi)\\ \sin (u +\phi) \end{array} \right)ds dt.
\end{align*}
Let $\delta \in (0, T)$ be so small that $supp\, \varphi \subset (\delta, T)$.
Using integration by parts, \eqref{periodicityproperty}, and  provided $\tau_{n} < \delta$ (which is true for $n$ sufficiently large), we can write
\begin{align*}
&\int_{0}^{T}\varphi(t)\int_{I} |(\tilde{U}_{n}+\phi)_{s}|^{p}  \left (\begin{array}{c} \cos (\tilde{U}_{n} +\phi)\\ \sin (\tilde{U}_{n} +\phi) \end{array} \right)ds dt\\
&=\int_{\delta}^{T}\varphi(t)\int_{I} |(\tilde{U}_{n}+\phi)_{s}|^{p}  \left (\begin{array}{c} \cos (\tilde{U}_{n} +\phi)\\ \sin (\tilde{U}_{n} +\phi) \end{array} \right) ds dt\\
&=\int_{\delta}^{T}\varphi(t)\int_{I} |(\tilde{U}_{n}+\phi)_{s}|^{p-2}(\tilde{U}_{n}+\phi)_{s}(\tilde{U}_{n}+\phi)_{s}  \left (\begin{array}{c} \cos (\tilde{U}_{n} +\phi)\\ \sin (\tilde{U}_{n} +\phi) \end{array} \right) ds dt\\
&=\int_{\delta}^{T}\varphi(t)\int_{I} (|(\tilde{U}_{n}+\phi)_{s}|^{p-2} (\tilde{U}_{n}+\phi)_{s})_{s} \left (\begin{array}{c} -\sin (\tilde{U}_{n} +\phi)\\ \cos (\tilde{U}_{n} +\phi) \end{array} \right)ds dt\\
&=\int_{\delta}^{T}\varphi(t)\int_{I} (|(\tilde{U}_{n}+\phi)_{s}|^{p-2} (\tilde{U}_{n}+\phi)_{s})_{s} \left (\begin{array}{c} -\sin (\tilde{U}_{n} +\phi) + \sin (u +\phi)\\ \cos (\tilde{U}_{n} +\phi) -\cos(u +\phi)\end{array} \right)ds dt\\
& \quad + \int_{\delta}^{T}\varphi(t)\int_{I} (|(\tilde{U}_{n}+\phi)_{s}|^{p-2} (\tilde{U}_{n}+\phi)_{s})_{s} \left (\begin{array}{c} -\sin (u(s,t) +\phi)\\ \cos(u(s,t) +\phi)\end{array} \right)ds dt\\
& =J_{1}+ J_{2} .
\end{align*}
We have that each component $J_{1}^{r}$ of $J_{1}= (J_{1}^{1}, J_{1}^{2})$ can be controlled by
\begin{align*}
|J_{1}^{r}| \leq C \|\varphi\|_{L^{2}(0,T)} \|\tilde{U}_{n} - u \|_{C^{0}([0,T] \times I)} \Big(\int_{\delta}^{T} \| |(\tilde{U}_{n}(s,t)+\phi)_{s}|^{p-2} (\tilde{U}_{n}(s,t)+\phi)_{s} \|_{W^{1,2}(I)}^{2} dt\Big )^{1/2}.
\end{align*}
The last integral is bounded (we have used this fact already in the proof of Lemma~\ref{lem:4.4}, see comments before \eqref{eq:4.18}). The uniform convergence of $\tilde{U}_{n}$ is proven in Lemma~\ref{lem:4.2}. So $J_{1} $ goes to zero as $n \to \infty$.
On the other hand using again \eqref{periodicityproperty} we can write
\begin{align*}
J_{2}=\int_{\delta}^{T}& \varphi(t)\int_{I} |(\tilde{U}_{n}(s,t)+\phi)_{s}|^{p-2} (\tilde{U}_{n}(s,t)+\phi)_{s}(u(s,t)+\phi)_{s} \left (\begin{array}{c} \cos (u(s,t) +\phi)\\ \sin(u(s,t) +\phi)\end{array} \right)ds dt
\end{align*}
and 
\begin{align}
J_{2} \to \int_{0}^{T}\varphi(t)\int_{I} |(u+\phi)_{s}|^{p}  \left (\begin{array}{c} \cos (u +\phi)\\ \sin (u +\phi) \end{array} \right)ds dt
\end{align}
by applying Lemma~\ref{lem:4.4bis} in each space component of $J_{2}$ with test function $\psi(s,t)= \varphi(t) \sin (u(s,t)+\phi(s))$ respectively $\psi(s,t)= -\varphi(t) \cos (u(s,t)+\phi(s))$ which belong to $L^{\infty}(0,T, W^{1,p}_{\rm per}(I))$ since $u \in L^{\infty}(0,T;W^{1,p}_{\rm per}(I))$.
\end{proof}

Next we show that the curves stay closed, i.e. \eqref{constraint} is fulfilled along the flow.
\begin{lem}\label{lem:closedness}
Let the assumptions of Theorem~\ref{thm:existence} hold.
Then we have that 
\begin{align}\label{quasi1.3}
\int_{I} \cos (u_{0}+\phi) ds = \int_{I} \cos (u +\phi) ds \qquad \text{ and } \qquad \int_{I} \sin (u_{0}+\phi) ds = \int_{I} \sin (u +\phi) ds
\end{align}
 holds on $[0,T)$.
\end{lem}
\begin{proof}  
We can repeat the calculations done in \eqref{derivataT}, use \eqref{bella} and the previous Lemma~\ref{lem:3.16} to infer that
 $\frac{d}{dt} \int_{I} T(u+\phi) ds =0 $ almost everywhere. The claim then follows by integration in time and using the continuity of $u$.
\end{proof}

\begin{proof}[Proof of Theorem { \ref{Theorem:1.1}}]
This is a direct consequence  of Theorem~\ref{thm:existence} and Lemma~\ref{lem:closedness}.
\end{proof}


\subsubsection{On regularity and uniqueness of weak solutions for \eqref{eq:P}  for $p\geq 2$}

To infer stronger regularity properties of the solution we will use Lemma~\ref{controlVel} and Lemma~\ref{lem:controlVelpBig}. Before we  give our statements, a few comments on the choice of the initial data are due. In both of the aforementioned  lemmata we need that the initial data $u_{0} \in W^{1,p}_{\rm per}(I) $ is such that the weak derivative $(|(u_{0}+\phi)_{s}|^{p-2}(u_{0}+\phi)_{s})_{s}$ exists and belongs to $L^{2}(I)$ and $[|(u_{0}+\phi)_{s}|^{p-2}(u_{0}+\phi)_{s}]_{0}^{L}=0$. Note that if $p=2$ then this is equivalent to saying that $u \in W^{2,2}_{\rm per}(I)$.
In Lemma~\ref{lem:controlVelpBig} the additional smallness condition on $u$ and Lemma~\ref{mehrreg} below yields the same conclusion: thus, from now on, we just ask for $u_{0} \in W^{2,2}_{\rm per}(I)$.

\begin{lem}\label{mehrreg} Assume $p\geq 2$. 
 Let $f \in W^{1,p}_{\rm per}(I)$ be a map such that $g:=|f_{s}|^{p-2}f_{s}$ admits weak derivative $g_{s} \in L^{2}(I)$, $g(0)=g(L)$, and $|f_{s}| \geq \alpha>0$ almost everywhere. Then $f \in W^{2,2}_{\rm per}(I)$.
Conversely if $f \in W^{2,2}_{\rm per}(I)$ with $|f_{s}| \geq \alpha>0$ on $I$ then $g \in W^{1,2}_{\rm per}(I) $.
We have that
$$ g_{s}=(p-1)|f_{s}|f_{ss}. $$
\end{lem}
\begin{proof}
 Since by assumption  $g \in L^{q}(I)$ with $q=p/(p-1) \leq 2$ and $g_{s} \in L^{2}(I)$, then $g \in W^{1,q}(I)$. By embedding theory $g$ is continuos and in particular $g \in L^{2}$. Hence $g \in W^{1,2}_{\rm per}(I)$. Then also $|g|=|f_{s}|^{p-1} \in W^{1,2}_{\rm per}(I)$ and $|g|\geq \alpha ^{p-1}>0$ in $I$. Therefore also $|g|^{1/(p-1)} =|f_{s}|\in W^{1,2}(I)$. But then $f_{s}$ has a constant sign by hypothesis hence $f \in W^{2,2}(I)$.
\end{proof}

We start by making a remark on the regularity of the weak solutions. This allows us to  infer that the energy decreases along the flow.
\begin{lem}\label{lem:3.18}
Let the assumptions of Theorem~\ref{thm:existence} hold.
Furthermore let one of the following assumptions be fulfilled:

(i) $p=2$, $u_{0} \in W^{2,2}_{\rm per}(I)$,

or

(ii) $p>2$, $u_{0} \in W^{2,2}_{\rm per}(I)$ and $u_{0}$ fulfills the smallness conditions \eqref{SC} given in Lemma~\ref{lem:controlVelpBig}.
Then have that, 
\begin{align*}
u \in H^{1}(0,T;W^{1,2}(I)), \qquad \lambda_{j} \in H^{1}(0,T), \quad j=1,2,
\end{align*} 
and the energy decreases along the flow, that is
$$ \frac{d}{dt} F_{p}(u)  \leq 0$$
holds for almost every time $t \in   (0,T)$.
\end{lem}
\begin{proof}
First of all notice that Lemma~\ref{controlVel} (case $p=2$) and Lemma~\ref{lem:controlVelpBig} (case $p>2$) yield the additional information that
 $$\int_{0}^{T}|(V_{n})_{s}|^{2} ds dt \leq C. $$

The regularity of the Lagrange multipliers follows now from  Lemma~\ref{lem:Lambda} and Lemma~\ref{lem:lambda2plus}.
Recalling \eqref{VH} in Lemma~\ref{lem:4.1} we infer that not only $V_{n}=\partial_{t} u_{n} \rightharpoonup u_{t}$ in $L^{2}(0,T; L^{2}(I))$ but also that
\begin{align*}
(V_{n})_{s}\rightharpoonup u_{ts}  \qquad \text{ in } \qquad L^{2}(0,T; L^{2}(I))
\end{align*}
In other words $u \in H^{1}(0,T;W^{1,2}(I))$. Then we can repeat the calculation performed in Lemma~\ref{energieabnahme} using now \eqref{bella} and Lemma~\ref{lem:closedness}. This gives the claim.
\end{proof}

Next we show  that under the assumptions of Lemma \ref{lem:controlVelpBig} the pointwise bounds derived in the mentioned lemma carry over to the solution.
\begin{lem}\label{lem:3.19}
Let the assumptions of Theorem~\ref{thm:existence} hold.
Furthermore let $p>2$, $u_{0} \in W^{2,2}_{\rm per}(I)$ and $u_{0}$ fulfills the smallness conditions \eqref{SC} given in Lemma~\ref{lem:controlVelpBig}. Then
$$ C \geq |(u+\phi)_{s}|^{p-1} \geq \frac{1}{4}|\phi_{s}|^{p-1}$$
holds on $I \times  [0,T]$.
\end{lem}
\begin{proof}
 Recalling the notation of Lemma~\ref{lem:4.1}, using Lemma~\ref{lemma3bis}, Lemma~\ref{mehrreg} and the additional information given by Lemma~\ref{lem:controlVelpBig} we have that 
\begin{align}\label{linftyb2}
\sup_{ t \in [0,T]} \| \partial_{s} u_{n} (\cdot, t) \|_{L^{\infty}(I)} \leq C, \qquad \text{ with } C=C(L,c_{*},p, \phi)
\end{align}
and
\begin{align}\label{linftyb3}
\sup_{ t \in [0,T]} \| \partial^{2}_{ss} (u_{n} +\phi)(\cdot, t) \|_{L^{2}(I)} \leq \sup_{ i \in \{0, \ldots,n \}} C(\phi,p) \|( |(u_{i,n} +\phi)_{s}|^{p-2} (u_{i,n} +\phi)_{s})_{s}\|_{L^{2}(I)} \leq C. 
\end{align}
Note that actually the constant does not depend on the initial value thanks to the first condition in \eqref{SC}: hence $C=C(L,p, \phi)$.
This yields
\begin{align}\label{pip2}
|\partial_{s} u_{n}(s_{2},t) - \partial_{s} u_{n}(s_{1},t)| \leq \left|\int_{s_{1}}^{s_{2}}   (u_{n})_{ss}(s, t)ds \right| \leq C |s_{2}-s_{1}|^{\frac{1}{2}}
\end{align}
with $C=C(L,p, \phi)$.
Next note that since $\partial_{s} u_{n}(s, \cdot)$ is absolutely continuous in $[\,0, T\,]$, 
for any $t_{1}$, $t_{2} \in [\,0, T\,]$ with $t_{1}< t_{2}$, we infer from H\"older's inequality and Fubini's theorem that  
\begin{align*}
\| \partial_{s}u_{n}(\cdot, t_{2}) - \partial_{s}u_{n}(\cdot, t_{1}) \|_{L^{2}(I)} 
 &= \left( \int_{I} \av{\int^{t_{2}}_{t_{1}} \dfrac{\pd^{2} u_{n}}{\pd s\pd t}(s,\tau) \, d\tau}^{2}\, ds \right)^{\frac{1}{2}} \\
 &\le \left( \int^{t_{2}}_{t_{1}} \int_{I} \av{\dfrac{\pd^{2} u_{n}}{ \pd s\pd t}(s,\tau)}^{2}\, ds d \tau \right)^{\frac{1}{2}} (t_{2}-t_{1})^{\frac{1}{2}}. 
\end{align*} 
According to Lemma~\ref{lem:controlVelpBig}, we have
\begin{align} \label{eq:4.4bis}
\int^{t_{2}}_{t_{1}} \int_{I} \av{\dfrac{\pd^{2} u_{n}}{\pd s \pd t}(s,\tau)}^{2}\, ds d \tau
 = \int^{t_{2}}_{t_{1}} \int_{I} \av{\partial_{s} V_{n}(s,\tau)}^{2} \, ds d \tau 
 \le C.  
\end{align}
Hence we obtain
\begin{align} \label{eq:4.5bis}
\| \partial_{s} u_{n}(\cdot, t_{2}) -\partial_{s} u_{n}(\cdot, t_{1}) \|_{L^{2}(I)} \le C (t_{2}-t_{1})^{\frac{1}{2}}. 
\end{align}

Fix $0 \le t_1 \le t_2 \le T$ arbitrarily and set $$\Gm_{s}(\cdot):= \partial_{s} u_n(\cdot, t_2) - \partial_{s} u_n(\cdot, t_1) \in W^{1,2}(I).$$ 
By an interpolation inequality (see for instance \cite[Thm. 5.9]{Adams})
we find 
$ 
\| \Gm_{s} \|_{L^\infty} \le C \| \Gm_{s} \|_{L^2}^{1/2} \|  \Gm_{s} \|_{W^{1,2}}^{1/2}. 
$
Using \eqref{linftyb2} and \eqref{linftyb3} yield
$
\| \Gm_{s} \|_{L^\infty} \le C \| \Gm_{s} \|_{L^2}^{1/2}.
$
By \eqref{eq:4.5bis} we infer
$\| \Gm_{s} \|_{L^\infty} \le C | t_2 - t_1 |^{1/4} $
with $C=C(L,p, \phi)$.
Hence, it follows  that 
\begin{align}\label{eq:4.12bis} 
\| \partial_{s } u_{n}(t_{2}) - \partial_{s} u_{n}(t_{1}) \|_{C^{0}} \le C | t_{2} - t_{1} |^{1/4}. \end{align}
From the above inequality and \eqref{pip2} it follows that
\begin{align}
|\partial_{s } u_{n}( s_{2}, t_{2}) - \partial_{s } u_{n}(s_{1}, t_{1})| \leq C (|t_{2}-t_{1}|^{1/4} + |s_{2}-s_{1}|^{1/2})
\end{align}
for any $(t_{i}, s_{i}) \in [0,T] \times [0,L]$, $i=1,2$.
Application of the Arzel\`a-Ascoli theorem yields (up to a subsequence) uniform convergence of $(u_{n})_{s}$ to $u_{s}$.
Next, fix $t \in (\,0, T\,]$ arbitrarily. 
Then there exists a family of intervals $\{ (\,(i_{n}-1) \tau_{n}, i_{n} \tau_{n} \,] \}_{n \in \N}$ such that 
$t \in (\,(i_{n}-1) \tau_{n}, i_{n} \tau_{n} \,]$. 
We deduce from \eqref{eq:4.12bis} that 
\begin{align*}
\| \partial_{s}\tilde{u}_{n}(t) - \partial_{s} u_{n}(t) \|_{C^{0}(I)} 
& = \| \partial_{s} u_{i_{n},n} - \partial u_{n}(t) \|_{C^{0}(I)} 
 = \| \partial_{s}u_{n}(i_{n} \tau_{n}) - \partial_{s}u_{n}(t) \|_{C^{0}(I)} \\
& \le C |i_{n} \tau_{n} - t|^{1/4} 
 \le C \tau_{n}^{1/4} \to 0 \quad \text{as} \quad n \to \infty . 
\end{align*} 
Since $\partial_{s} u_{n} \to u_{s}$ in $C^{0}([0,T] \times I)$ by the argument shown above and $t$ was arbitrarily chosen, we infer that $\partial_{s } \tilde{u}_{n} \to \partial_{s } u$ in $C^{0}([0,T] \times I)$. 
This together with Lemma~\ref{lem:controlVelpBig}  yields the claim.
\end{proof}


Next we show that the map of Theorem~\ref{Theorem:1.1} is  unique.

\begin{lem}\label{lem:3.20}

Let the assumptions of Theorem~\ref{thm:existence} hold.
Furthermore let one of the following assumptions be fulfilled:

(i) $p=2$ and $u_{0}^{u}, u_{0}^{w} \in W^{1,2}_{\rm per}(I)$,

or

(ii) $p>2$ and $u_{0}^{u},u_{0}^{w} \in W^{2,2}_{\rm per}(I)$ and both $u_{0}^{u}$ and $u_{0}^{w}$ fulfill the smallness conditions \eqref{SC} given in Lemma~\ref{lem:controlVelpBig}.

Assume that $u,w \in L^{\infty}(0,T; W^{1,p}_{\rm per}(I)) \cap H^{1}(0,T; L^{2}(I))$ are two weak solutions of \eqref{eq:P} on 
$[0,T)$   with Lagrange multipliers $\lambda_{i}^{u}=\lambda_{i}^{w}$ for $i=1,2$ and  initial data $u_{0}^{u}$ resp. $u_{0}^{v}$. Then
\begin{align}
\int_{0}^{T} |\lambda_{j}^{u} -\lambda_{j}^{w}|^{2} dt & \leq C(T) \|u_{0}^{u}-u_{0}^{w}\|_{L^{2}(I)}^{2},\\
\sup_{\tau \in (0,T)}\| (u-w)(\tau)\|^2_{L^2(I)} + C \int_{0}^T \| \partial_s(u-w) \|^2_{L^2(I)} \, dt &\leq C(T) \|u_{0}^{u}-u_{0}^{w}\|_{L^{2}(I)}^{2}.
\end{align}
In particular if $u_{0}^{w}=u_{0}^{u}$, 
then $u=w$ and $\lambda_{i}^{u}=\lambda_{i}^{w}$ for $i=1,2$. 
\end{lem}

\begin{proof}
Fix $0<\tau<T$ and 
let $\eta_{\epsilon} \in C^\infty(0,T)$ be a family of functions such that 
\begin{align*}
\eta_{\epsilon}(t) \equiv 1 \quad \text{in} \quad [0, \tau], \qquad \eta_{\epsilon}(t) \equiv 0 \quad \text{in} \quad [\tau+\epsilon, T), \qquad 0 \le \eta_{\epsilon}(t) \le 1 \quad \text{in} \quad [0, T],  
\end{align*}
where $0 < \epsilon < T-\tau$, and with $\eta_{\epsilon} \to \chi_{[0,\tau]}$ as $\epsilon $ goes to zero. 
From \eqref{eq:1.2} we have 
\begin{equation} \label{eq:4.15}
\begin{split} 
& \int^T_0 \int_I \left[ \left\{ \partial_t u - \lambda_1^{u} \sin(u+\phi) + \lambda_2^{u} \cos{(u+\phi)} \right\} \varphi \right. \\ 
                & \qquad \qquad \left. + \left( |\partial_s(u+\phi)|^{p-2} \partial_s(u+\phi) \right) \partial_s \varphi \right] \, ds dt = 0,  
\end{split}
\end{equation}
and 
\begin{equation} \label{eq:4.16}
\begin{split}
& \int^T_0 \int_I \left[ \left\{ \partial_t w - \lambda_1^{w} \sin(w+\phi) + \lambda_2^{w} \cos{(w+\phi)} \right\} \varphi \right. \\ 
                & \qquad \qquad \left. + \left( |\partial_s(w+\phi)|^{p-2} \partial_s(w+\phi) \right) \partial_s \varphi \right] \, ds dt = 0,  
\end{split}
\end{equation}
for any $\varphi \in L^\infty(0,T; W^{1,p}_{\rm per}(I))$. 
Choosing $\varphi=(u-w)\eta_{\epsilon}$ in \eqref{eq:4.15} and \eqref{eq:4.16}, subtracting and letting $\epsilon \to 0$, we deduce that 
\begin{equation} \label{eq:4.17bis}
\begin{split} 
0&= \int^\tau_0 \int_I  \partial_t (u-w)(u-w) \, dsdt \\ 
& \,\,\,\,\, - \int^\tau_0 \int_I [ \lambda_1^{u} \sin(u+\phi) - \lambda_1^{w} \sin(w+\phi)](u-w) \, dsdt \\
& \,\,\,\,\, + \int^\tau_0 \int_I [ \lambda_2^{u} \cos{(u+\phi)} - \lambda_2^{w} \cos{(w+\phi)} ](u-w) \, dsdt  \\
& \,\,\,\,\, + \int^\tau_0 \int_I [ |\partial_s(u+\phi)|^{p-2} \partial_s(u+\phi) - |\partial_s(w+\phi)|^{p-2} \partial_s(w+\phi)] \partial_s(u-w) \, ds dt \\
& =: J_1 + J_2 + J_3 + J_4.    
\end{split}
\end{equation}
First we have 
\begin{align} \label{eq:4.18bis}
J_1 = \dfrac{1}{2} \int^\tau_0 \dfrac{d}{dt} \int_I  \| (u-w)(t)\|^2_{L^2(I)} \, dt 
       = \dfrac{1}{2} \| (u-w)(\tau)\|^2_{L^2(I)} -\frac{1}{2} \|u_{0}^{u}-u_{0}^{w}\|_{L^{2}(I)}^{2}. 
\end{align}
 Recalling that $\lambda_{j}^{u}=\tilde{\lambda}_{j}(u)$ by Lemma~\ref{lem:3.16},
it follows from \eqref{LIPlambda} in case $p=2$  or from Lemma~\ref{lem:3.19} and the same arguments employed to derive \eqref{est-lambda-forte} if $p>2$ that 
\begin{equation} \label{eq:4.19bis}
\begin{split}
|J_2| &\le \int^\tau_0 \int_I | \tilde{\lambda}_1(u) - \tilde{\lambda}_1(w) | |u-w| \, dsdt \\ 
         &\qquad + \int^\tau_0 \int_I | \tilde{\lambda}_1(w) | | \sin{(u+\phi)} - \sin{(w+\phi)} | |u-w| \, dsdt \\
                  &\le C \int^\tau_0 \| u-w \|^2_{L^2(I)} \, dt + C\int^\tau_0 \| \partial_s(u-w) \|_{L^2(I)}  \| u-w \|_{L^2(I)}\, dt. 
\end{split}
\end{equation}
Along the same line as in \eqref{eq:4.19bis}, we find 
\begin{align} \label{eq:4.20bis}
|J_3| \le C \int^\tau_0 \| u-w \|^2_{L^2(I)} \, dt + \int^\tau_0 \| \partial_s(u-w) \|_{L^2(I)}  \| u-w \|_{L^2(I)} \, dt. 
\end{align}
Thanks to the condition $p \ge 2$, Lemma~\ref{lem:3.19} (for the case $p>2$) we infer from the first inequality in Lemma \ref{Lindquist} that 
\begin{align} \label{eq:4.21bis}
J_4 \ge C \int^\tau_0 \| \partial_s(u-w) \|^2_{L^2(I)} \, dt, 
\end{align}
where the constant depends only on $p$ and $\phi$. 
Combining \eqref{eq:4.17bis} with \eqref{eq:4.18bis}, \eqref{eq:4.19bis}, \eqref{eq:4.20bis} and \eqref{eq:4.21bis}, we have 
\begin{equation} \label{eq:4.22bis}
\begin{split}
& \| (u-w)(\tau)\|^2_{L^2(I)} + C \int^\tau_0 \| \partial_s(u-w) \|^2_{L^2(I)} \, dt \\
& \qquad \le C \int^\tau_0 \left( \| u-w \|^2_{L^2(I)} + \|\partial_s(u-w) \|_{L^p(I)}  \| \| u-w \|_{L^2(I)} \right) \, dt
+ C \|u_{0}^{u}-u_{0}^{w}\|_{L^{2}(I)}^{2}. 
\end{split}
\end{equation}
Using Young's inequality, we reduce \eqref{eq:4.22bis} to 
\begin{equation} \label{eq:4.23}
\begin{split}
& \| (u-w)(\tau)\|^2_{L^2(I)} \le C \int^\tau_0 \| u-w \|^2_{L^2(I)} \, dt + C\|u_{0}^{u}-u_{0}^{w}\|_{L^{2}(I)}^{2}. 
\end{split}
\end{equation}
Since $\tau$ could be chosen arbitrarily,  we infer from Gronwall's inequality  that 
\begin{align*} 
 \sup_{\tau \in (0,T)}\| (u-w)(\tau)\|^2_{L^2(I)} + C \int_{0}^T \| \partial_s(u-w) \|^2_{L^2(I)} \, dt \leq C(T) \|u_{0}^{u}-u_{0}^{w}\|_{L^{2}(I)}^{2}. 
\end{align*}
The same arguments used above give then
\begin{align*}
\int_{0}^{T} |\lambda_{j}^{u} -\lambda_{j}^{w}|^{2} dt \leq  C \int_{0}^{T} \| u-w\|^2_{L^2(I)} +\| \partial_s(u-w) \|^2_{L^2(I)} \, dt \leq C(T) \|u_{0}^{u}-u_{0}^{w}\|_{L^{2}(I)}^{2}.
\end{align*}
Hence $u=v$
 and  $\lambda_{j}^{u}-\lambda_{j}^{w}$ for $j=1,2$ if the initial data coincide.
\end{proof}. 

\section{Long time existence}\label{sec:4}

\begin{proof}[Proof of Theorem { \ref{Theorem:1.2}}]
Since the energy  decreases along the flow  by Lemma~\ref{lem:3.18} we have that $F_{p}(u(T)) \leq F_{p}(u_{0})$ and Theorem~\ref{thm:3.2} can be applied again under exactly the same conditions at time $t=T$. That is, we take  $u(T)$ as the new initial value. Note that by Theorem~\ref{thm:existence} the solution is also as regular as required by Lemma~\ref{controlVel}. In fact, it is even as regular as required  by Lemma~\ref{lem:controlVelpBig}. In this case, however, the decrease of the energy, that is  the gradient structure of the problem, is not sufficient to guarantee that the second condition in \eqref{SC} is fulfilled also at time $T$. Thus, since $p=2$, we find a weak solution to the problem \eqref{eq:P} on the time interval $(T,2T)$. By iterating this procedure we obtain the claim.
\end{proof}



\bibliography{ref}
\bibliographystyle{acm}


\end{document}